\documentclass[11pt]{article}
\usepackage{setspace}
\usepackage{mathrsfs}
\usepackage{mathrsfs}
\usepackage{mathrsfs}
\usepackage{graphicx}
\usepackage{epstopdf}
\usepackage{multirow}
\usepackage{amsmath}
\usepackage{cite}
\usepackage[utf8]{inputenc}
\renewcommand {\baselinestretch} {1.3}

\usepackage{amsfonts}
\usepackage{amssymb}
\usepackage{amsmath}
\usepackage{amsthm}
\usepackage[colorlinks=true]{hyperref}
\usepackage{hyperref}
\hypersetup{linkcolor=blue}
\newtheorem{thm}{Theorem}[section]

\newtheorem{ass}[thm]{Assumption}

\newtheorem{coro}[thm]{Corollary}

\newtheorem{defn}[thm]{Definition}

\newtheorem{lem}[thm]{Lemma}

\newtheorem{prop}[thm]{Proposition}
\newtheorem{remk}[thm]{Remark}

\numberwithin{equation}{section}

\newcommand{\epsi}{\varepsilon}
\begin{document}
\begin{spacing}{0.5}
\end{spacing}
\title{{\Large\bf {{Basic Reproduction Ratios for Almost Periodic Non-Densely Defined Cauchy Problems}}}
\thanks{
This work is supported by the Beijing Natural Science Foundation (No. 1254047) and the National Natural Science Foundation of China (No. 12171039).}}
\author{{\normalsize Jiawei Huo$^{(a),}${\thanks{Corresponding author:  jiaweihuo@buct.edu.cn  (J. Huo)}} \,, \, Rong Yuan$^{(b)}$,  \,\,\,\,}\\}
\date{$^{(a)}$ College of Mathematics and Physics, Beijing University of Chemical Technology,  Beijing, 100029, People’s Republic of China.\\
$^{(b)}$ School of Mathematical Sciences, Beijing Normal University, Beijing, 100875, \\People’s Republic of China.\\}
\maketitle
\begin{minipage}{14cm} {\bf Abstract:} The theory of the basic reproduction ratio $R_0$ is established for the almost periodic non-densely defined Cauchy problems. Firstly, we present some dynamical properties of the linear almost periodic non-densely defined Cauchy problems. Then we define the evolution semigroups on the space of almost periodic functions, investigate their infinitesimal generators and define the basic reproduction ratio $R_0$. Finally, our developed theory is applied to age-structured models and functional differential systems.

\end{minipage}
 \maketitle
\numberwithin{equation}{section}
\newtheorem{theorem}{Theorem}[section]
\newtheorem{lemma}[theorem]{Lemma}
\newtheorem{proposition}[theorem]{Proposition}
\newtheorem{corollary}[theorem]{Corollary}
\newtheorem{assumption}[theorem]{Assumption}

\section{Introduction}
\indent\indent The basic reproduction ratio $R_0$, as an important concept, is widely used in population dynamics and epidemiology. In population dynamics, $R_0$ is the total number of expected offspring an average individual would produce during its lifetime. In epidemiology, $R_0$ is the expected number of secondary cases produced {in a population of completely susceptible individuals. Usually, when $R_0$ is less than one, the population or disease eventually becomes extinct, whereas persistence occurs when $R_0$ is greater than one.}

The theory of the basic reproduction ratio $R_0$ in different environments has been established by many scholars. In 1990, Diekmann, Heesterbeek and Metz studied an epidemic model with a heterogeneous population and introduced a next generation operator method to define the basic reproduction number $R_0$ \cite{Diekmann90}, and then van den Driessche and Watmough extended the definition of $R_0$ to the autonomous compartmental epidemic models \cite{van}. For non-autonomous models, Wang and Zhao defined the basic reproduction ratio and studied the threshold dynamics of the periodic compartmental models \cite{Wang08}. Later, Thieme investigated the relation between spectral bound and reproduction ratios and extended $R_0$ to models with infinite-dimensional state spaces and time heterogeneity \cite{Thieme09}. If the models are considered in an almost periodic situation, Wang and Zhao obtained similar results by evolution semigroups \cite{Wang13}. Until now, the basic reproduction ratio $R_0$ is always defined as the spectral radius of the next generation operator $FV^{-1}$ (i.e. $R_0:=r(FV^{-1})$) and has been applied in various autonomous and non-autonomous functional differential systems\cite{Zhao17,Qiang20}  reaction-diffusion systems\cite{Wang12,Zhang21,Wang23} and nonlocal diffusion systems\cite{Lin,Feng,Yang}.

{ However, in many autonomous age-dependent population and epidemic models, the next generation operator is typically represented as an integral operator, and the basic reproduction ratio $R_0$ is defined as its spectral radius.} For periodic age-dependent models, the next generation operator is also obtained by Bacaër and Guernaoui\cite{Bacaër} and Inaba \cite{Inaba19}. It is well-known that the age-structured models can be reformulated as a Cauchy problem in a non-dense domain. This provides another method to define the basic reproduction ratios for the age-structured models and other models that can be transferred to a non-densely defined Cauchy problem. Thieme studied an age-structured population model with spatial diffusion and defined the basic reproduction ratio $R_0=r(-F\mathcal B^{-1})$, where $F,\mathcal B$ are non-densely defined operators\cite{Thieme09}. Subsequently, Huo, Huo and Yuan extended it to n-dimensional cooperative age-structured epidemic models with spatial diffusion and degenerate diffusion \cite{Huo23}. Recently, Djidjou-Demasse, Goudiaby and Seydi proposed a general framework to define $R_0$ for the periodic non-densely defined Cauchy problems and applied their results to several biological models \cite{Djidjou-Demasse}.

It is well known that seasonality due to changing weather conditions is a common external environmental factor that affects population dynamics and annual trends in the spread of infectious diseases\cite{Liu17}. In periodic population and epidemic models, all the involved coefficients are assumed to be periodic and share a common period. However, as noted in \cite{Altizer}, seasonal changes are cyclic, largely predictable and arguably represent the strongest and most ubiquitous source of external variation influencing human and natural systems. Thus, the death rate, recovery rate and birth rate are more reasonable to have complex oscillation. Even if fluctuations in nature are periodic, they also do not always share a common period. In particular, if the periods of these periodic coefficients have no common integer multiple, then the model is not a periodic system. Almost periodic functions, a generalization of periodic functions, are more suitable to characterize the fluctuations in nature\cite{Bezandry,Diagana}. Our purpose in this article is to give a general framework for defining the basic reproduction ratio $R_0$ for the almost periodic non-densely defined Cauchy problems.

To the best of our knowledge, almost periodic non-densely defined Cauchy problems have been studied by many scholars in the past \cite{Cuevas,Arendt99,Amir,Afoukal,Boulite}. However, most works focused on the existence of almost periodic solutions, with little attention being paid to the dynamical behaviors of the non-densely defined Cauchy problems, especially the threshold properties. In this paper, we define the basic reproduction ratio $R_0$ for the almost periodic non-densely defined Cauchy problems by the spectral radius of an integral operator, i.e., $R_0:=r(\mathscr L)$. Moreover, we show that $r(\mathscr L)$ is equal to $r(-F\mathcal B^{-1})$, where $F$ and $\mathcal B$ are non-densely defined operators.  This implies that the basic reproduction ratio $R_0$ of the non-densely defined Cauchy problems can be defined in the same way as ODE, reaction-diffusion systems and nonlocal diffusion systems. Furthermore, we prove that $R_0-1$ has the same sign as the exponential growth bound of the linear non-densely defined Cauchy problem, which is the core property of the basic reproduction ratio in many models \cite{Wang13,Wang12,Qiang20,Thieme09,Zhao17}. It is worth noting that age-structured models, functional differential systems and some parabolic equations can be formulated as a non-densely defined Cauchy problem. Thus, the proposed approach can be applied to the corresponding models in the almost periodic situation and define their basic reproduction numbers. 

The rest of this paper is organized as follows. In Section 2, we present some dynamical properties of the almost periodic non-densely defined Cauchy problems. In Section 3, we define the evolution semigroups on the space of almost periodic functions and investigate their infinitesimal generators. Moreover, we define the basic reproduction ratio $R_0$ and prove that $R_0-1$ has the same sign as the exponential growth bound of the corresponding linear non-densely defined Cauchy problem. In Section 4, our developed theory is applied to age-structured models and functional differential systems.

\section{Linear almost periodic non-densely defined Cauchy problems}
\indent\indent First, we give some definitions of the almost periodic function and the Hille-Yosida operator.
\begin{defn}\label{DE2.1}\cite{Corduneanu,Fink}
Let $(X,d)$ be a metric space. A function $f\in C(\mathbb R,X)$ is said to be \textbf{almost periodic} if for any $\epsi>0$, the set
\begin{equation}\label{2.1}
\mathcal T(f,\epsi)=\{s\in\mathbb R:d(f(t+s),f(t))< \epsi ,\forall t\in\mathbb R\}
\end{equation}
is a relatively dense subset of $\mathbb R$, that is, there is a positive constant $l>0$ (inclusion length) such that $[c,c+l]\bigcap \mathcal T(f,\epsi)\neq \emptyset, \forall c\in\mathbb R$. 
\end{defn}

We define the space of almost periodic functions by
$$
AP(\mathbb R,X):=\{f\in C(\mathbb R,X): f \text{ is an almost periodic function}\}.
$$
Then $AP(\mathbb R,X)$ is a Banach space equipped with supremum norm $\|\cdot\|$.

{\begin{defn}\label{DE2.2}
Let $X$ be a Banach space and $A:D(A)\subset X\to X$ be a linear operator. We say that $(A,D(A))$ is a \textbf{Hille-Yosida operator} on $X$ if there exist $\omega_A\in\mathbb R$ and a positive constant $M_A\ge 1$ such that $(\omega_A,+\infty)\subset \rho(A)$ ($\rho(A)$ is the resolvent set of $A$) and 
\begin{equation}\label{2.2}
\|(\lambda I-A)^{-n}\| \leq \frac{M_A}{(\lambda-\omega_A)^n}, \quad \forall\ \lambda>\omega_A, n\in\mathbb N^+.
\end{equation}
\end{defn}
}

{ In the mathematical study of biological systems, many population and epidemic models can be formulated as abstract Cauchy problems with non-dense domains, such as age-structured models and functional differential equations. In such models, the operator $A$ typically describes the intrinsic evolution of individuals, such as aging and spatial diffusion. In addition, $F(t)$ represents the production of new individuals or new infections and $V(t)$ describes mortality, recovery or other transition processes.}  Motivated by these applications, we study the basic reproduction ratio $R_0$ for the following linear almost periodic non-densely defined Cauchy problem
\begin{equation}\label{2.3}
\left\{ {\begin{array}{*{20}{l}}
\frac{du(t)}{dt}=Au(t)+F(t)u(t)-V(t)u(t), \quad t>t_0,\\
u(t_0)=u_0\in X_0=\overline{D(A)}.
\end{array}} \right.
\end{equation}
where $A:D(A)\subset X\to X$ is a Hille-Yosida linear operator (possibly unbounded, non-densely defined) on Banach space $(X,\|\cdot\|)$. 

Let us introduce the part $A_0$ of $A$ in $X_0=\overline{D(A)}$:
\begin{equation}\label{2.4}
A_0=A \text{ on } D(A_0)=\{\psi\in D(A), A \psi\in X_0   \}.
\end{equation}
Since $A$ is a Hille-Yosida operator on $X$, it follows from \cite[Lemmas 2.1 and 2.2]{Magal09MAMS} that $\rho(A)=\rho(A_0)$, $\overline{D(A)}=\overline{D(A_0)}$ and $A_0$ is a Hille-Yosida linear operator with dense domain $X_0$. By Hille-Yosida theorem, $A_0$ generates a $C_0-$semigroup $\{T_{A_0}(t)\}_{t\ge0}\subset \mathcal L(X_0)$ and satisfies the following estimate
\begin{equation}\label{2.5}
\|T_{A_0}(t)\|_{\mathcal L(X_0)}\le M_A e^{\omega_A t}, \quad t\ge 0,
\end{equation}
where $M_A$ and $\omega_A$ are given in Definition \ref{DE2.2}. Throughout this paper, unless otherwise stated, we assume that the Banach spaces $X$ and $X_0$ have normal and generating positive cones
$X_+$ and $X_0^+$ respectively. In addition, the maps $t\in \mathbb R\mapsto F(t)\in \mathcal L(X_0,X)$, $t\in \mathbb R\mapsto V(t)\in \mathcal L(X_0,X)$ are almost periodic and continuous in the operator norm topology. For the non-densely defined Cauchy problem \eqref{2.3}, we make the following assumptions.

\begin{ass}\label{ASS2.3}
For the non-densely defined Cauchy problem \eqref{2.3}, we assume that
\begin{itemize}
		\item[{\rm (i)}] $A:D(A)\subset X_0\to X$ is resolvent positive, i.e., the resolvent set of $A$, $\rho(A)$, contains a ray $(\omega,+\infty)$ and $(\lambda I-A)^{-1}$ is a positive operator for all $\lambda> \omega$.
\item[{\rm (ii)}] There exists $\lambda_1>\omega_A$ such that $\lambda_1 \phi-V(t)\phi\in X_+$ for all $\phi\in X_{0}^+$ and $t\in \mathbb R$.
\item[{\rm (iii)}] For each $t\in \mathbb R$ and each $\phi\in X_{0}^+$, we have $F(t)\phi\in X_+$.
\end{itemize}
\end{ass}

Note that almost periodic functions are bounded and uniformly continuous \cite[Theorem 1.13]{Fink}, it follows from \cite[Proposition 4.1]{Magal09ADE} that the non-densely defined Cauchy problem \eqref{2.3} admits an exponentially bounded evolution family $\{U(t,s)\}_{t\ge s}$ on space $X_0$. This implies that the solution $u(t)$ of the Cauchy problem \eqref{2.3} satisfies
\begin{equation}\label{2.6}
u(t)=U(t,t_0)u_0,\quad t\ge t_0.
\end{equation}
Moreover, the exponential growth bound $\omega(U)$ of the evolution family $U(t,s)$ is defined by
$$
\omega (U) = \inf \{ \hat{\omega} :\exists M \ge 1:\forall s \in \mathbb{R},t \ge s:\left\| {U(t,s)} \right\| \le M{e^{\hat{\omega} (t - s)}}\}.$$

Next, we consider another non-densely defined Cauchy problem
\begin{equation}\label{2.7}
\left\{ {\begin{array}{*{20}{l}}
\frac{dv(t)}{dt}=Av(t)-V(t)v(t), \quad t>t_0,\\
v(t_0)=v_0\in X_0=\overline{D(A)},
\end{array}} \right.
\end{equation}
where $A$ and $V(t)$ satisfy the same assumptions as \eqref{2.3}. By a similar method, we know that the Cauchy problem \eqref{2.7} admits an exponentially bounded evolution family $\{\Gamma(t,s)\}_{t\ge s}$ on space $X_0$. { Since $A$ is a Hille-Yosida operator and $V(t)$ is bounded and continuous, for the solution $v(t)$ of the Cauchy problem \eqref{2.7}, the map $f(t):=-V(t)v(t)$ is continuous. Applying \cite[Theorem 1.5]{Magal21} to the nonhomogeneous problem
$$
\frac{dv(t)}{dt}=Av(t)+f(t)
$$
with $f(t)=-V(t)v(t)$, we obtain the following variations of constants formula 
}
\begin{equation}\label{2.8}
v(t)=\Gamma(t,t_0)v_0=T_{A_0}(t-t_0)v_0-\!\!\lim_{\lambda\to +\infty}\int_{t_0}^t T_{A_0}(t-s)\lambda(\lambda I-A)^{-1} V(s)\Gamma(s,t_0)v_0ds,\forall v_0\in X_0,
\end{equation}
where the limit exists in $X_0$. Moreover, the convergence in \eqref{2.8} is uniform with respect to $t,t_0\in I$ for each compact interval $I\subset \mathbb R$.

In the following, we add another assumption on the evolution family $\{\Gamma(t,s)\}_{t\ge s}$.
\begin{ass}\label{ASS2.4}
The evolution family $\{\Gamma(t,s)\}_{t\ge s}$ satisfies
$$
\omega(\Gamma)<0,
$$
where $\omega(\cdot)$ represents the exponential growth bound.
\end{ass}
\begin{remk}\label{RE2.5}
Assumption \ref{ASS2.4} implies that evolution family $\{\Gamma(t,s)\}_{t\ge s}$ is exponentially stable. In fact, this assumption is classical, common and crucial in epidemic models, for example, see \cite[Assumption (A2)]{Qiang},\cite[Theorem 3.17]{Thieme09}, \cite[Assumption (A7)]{Wang13}.
\end{remk}

\begin{thm}\label{TH2.6}
Let Assumption \ref{ASS2.3} be satisfied. The mild solution $u(t)\in C([t_0,\infty),X_0)$ of the Cauchy problem \eqref{2.3} has the following form
\begin{equation}\label{2.9}
u(t)=U(t,t_0)u_0=\Gamma(t,t_0)u_0+\lim_{\lambda\to +\infty}\int_{t_0}^t \Gamma(t,s)\lambda (\lambda I-A)^{-1}F(s)U(s,t_0)u_0ds, \quad \forall u_0\in X_0,
\end{equation}
where the limit exists in $X_0$. Moreover, the convergence in \eqref{2.9} is uniform with respect to $t,t_0\in I$ for each compact interval $I\subset \mathbb R$. Furthermore, $u(t)$ satisfies the following properties
\begin{itemize}
\item[{\rm (i)}] $u(t)=U(t,t_0)u_0$ is a continuous function of $(t,t_0,u_0)$.
\item[{\rm (ii)}] $u(t)$ depends continuously on functions $F(t)$ and $V(t)$.
\end{itemize}
\end{thm}
\begin{proof}
Using \cite[Theorem 1.6]{Magal21} and the evolution family $\{\Gamma(t,s)\}_{t\ge s}$, we can easily obtain formula \eqref{2.9}. In addition, by \cite[Theorem 3.2]{Thieme90}, it is clear that property (i) holds. Finally, we show property (ii). We consider the following non-densely defined Cauchy problem
\begin{equation}\label{2.10}
\left\{ {\begin{array}{*{20}{l}}
\frac{dw(t)}{dt}=Aw(t)+\mathcal F(t)w(t)-\mathcal V(t)w(t), \quad t>t_0,\\
w(t_0)=u_0.
\end{array}} \right.
\end{equation}
where $\mathcal F(t)$ and $\mathcal V(t)$ admit the same assumptions as almost periodic functions $F(t)$ and $V(t)$. Based on the definition of the integrated solution of the non-densely defined Cauchy problems\cite[Definition 1.2]{Thieme90}, the integrated solutions of \eqref{2.3} and \eqref{2.10} have the following form
$$
u(t)=u_0+A\int_{t_0}^t u(s)ds+\int_{t_0}^t (F(s)-V(s))u(s)ds,
$$
and
$$
w(t)=u_0+A\int_{t_0}^t w(s)ds+\int_{t_0}^t (\mathcal F(s)-\mathcal V(s))w(s)ds.
$$
Then we obtain
$$
u(t)-w(t)=A\int_{t_0}^t(u(s)-w(s))ds+\int_{t_0}^t (F(s)-V(s))u(s)-({\mathcal F}(s)- {\mathcal V}(s))w(s)ds.
$$
Due to \cite{Kellermann} or \cite[Theorem 2.1]{Huo25}, we have the following estimate
$$
\|u(t)-w(t)\|\le  M_{A} \int_{t_0}^t e^{\omega_{A}(t-s)}\|(F(s)-V(s))u(s)-({\mathcal F}(s)-{\mathcal V}(s))w(s)\|ds.
$$
and hence,
\begin{align*}
    \|u(t)-w(t)\|&\le  M_{A} \int_{t_0}^t  e^{\omega_{A}(t-s)}\| (F(s)-{\mathcal F}(s))u(s)\|ds+ M_{A}\int_{t_0}^t  e^{\omega_{A}(t-s)}\|{\mathcal F}(s)(u(s)-w(s))\|ds\\
&\quad + M_{A} \int_{t_0}^t  e^{\omega_{A}(t-s)}\| ({\mathcal V}(s)-V(s))w(s)\|ds+ M_{A}\int_{t_0}^t  e^{\omega_{A}(t-s)}\|{V}(s)(w(s)-u(s))\|ds.
\end{align*}
Applying $e^{-\omega_A(t-t_0)}$ on both sides of the above equation, we obtain
\begin{align*}
&\quad e^{-\omega_A(t-t_0)}\|u(t)-w(t)\| = \|e^{-\omega_A(t-t_0)}u(t)-e^{-\omega_A(t-t_0)}w(t)\| \\
&\le  M_{A} \int_{t_0}^t  e^{-\omega_{A}(s-t_0)}\| (F(s)-{\mathcal F}(s))u(s)\|ds+ M_{A}\int_{t_0}^t  e^{-\omega_{A}(s-t_0)}\|{\mathcal F}(s)(u(s)-w(s))\|ds\\
&\quad + M_{A} \int_{t_0}^t  e^{-\omega_{A}(s-t_0)}\| ({\mathcal V}(s)-V(s))w(s)\|ds+ M_{A}\int_{t_0}^t  e^{-\omega_{A}(s-t_0)}\|{ V}(s)(w(s)-u(s))\|ds.
\end{align*}
By setting $\overline u(t)=e^{-\omega_A(t-t_0)}u(t)$ and  $\overline w(t)=e^{-\omega_A(t-t_0)}w(t)$, then we have
\begin{align*}
\|\overline u(t)-\overline w(t)\|&\le  M_{A} \int_{t_0}^t \| (F(s)-{\mathcal F}(s))\overline u(s)\|ds+ M_{A}\int_{t_0}^t  \|{\mathcal F}(s)(\overline u(s)-\overline w(s))\|ds\\
&\quad + M_{A} \int_{t_0}^t \| ({\mathcal V}(s)-V(s))\overline w(s)\|ds+ M_{A}\int_{t_0}^t \|{V}(s)(\overline w(s)-\overline u(s))\|ds.
\end{align*}
Recall that almost periodic functions are bounded, there exists a constant $K_*>0$ such that $\|\mathcal F(t)\|\le K_*$ and $\|V(t)\|\le K_*$. Thus, we have
$$
\begin{array}{ll}
\|\overline u(t)-\overline w(t)\|&\le  M_{A} \int_{t_0}^t \| (F(s)-{\mathcal F}(s))\overline u(s)\|ds+ M_{A}K_*\int_{t_0}^t  \|(\overline u(s)-\overline w(s))\|ds\\
&\quad + M_{A} \int_{t_0}^t \| ({\mathcal V}(s)-V(s))\overline w(s)\|ds+ M_{A}K_*\int_{t_0}^t \|(\overline w(s)-\overline u(s))\|ds.
\end{array}
$$
Applying a generalized Gronwall's inequality (\cite[Lemma 1.3.1]{Hale}), we have
$$
\|\overline u(t)-\overline w(t)\|\le \left(  M_{A} \int_{t_0}^t \| ( F(s)-{\mathcal F}(s))\overline u(s)\|ds+M_{A} \int_{t_0}^t \| ({\mathcal V}(s)- V(s))\overline w(s)\|ds \right)e^{M_{A}K_*(t-t_0)}.
$$
Furthermore, we obtain
$$
\|u(t)-w(t)\|\le \left(  M_{A} \int_{t_0}^t \| ( F(s)-{\mathcal F}(s))\overline u(s)\|ds+M_{A} \int_{t_0}^t \| ({\mathcal V}(s)- V(s))\overline w(s)\|ds \right)e^{(M_{A}K_*+\omega_A)(t-t_0)}.
$$
\end{proof}

\begin{remk}
In fact, \eqref{2.9} is the variation of constants formula of the non-densely defined evolution equations. It is worth mentioning that Theorem \ref{TH2.6} does not need Assumption \ref{ASS2.4}. Moreover, Theorem \ref{TH2.6} still holds when $F(t)$ and $V(t)$ are only bounded and uniformly continuous functions, instead of almost periodic functions.
\end{remk}

\section{Basic reproduction ratios}
\indent\indent In this section, we define two evolution semigroups on the space of almost periodic functions. Moreover, we investigate their infinitesimal generators. Furthermore, the basic reproduction ratio $R_0$ is defined for the almost periodic non-densely defined Cauchy problems.
\subsection{Evolution semigroups on the space of almost periodic functions}
\indent\indent By the evolution families $\{U(t,s)\}_{t\ge s}$ and $\{\Gamma(t,s)\}_{t\ge s}$, we define two one-parameter families $\{\mathbb U(t)\}_{t\ge0}$ and  $\{\mathbb V(t)\}_{t\ge0}$ as follows
\begin{equation}\label{3.1}
[\mathbb U(t)\phi](s)=U(s,s-t)\phi(s-t),\quad \forall\phi\in AP(\mathbb R,X_0)
\end{equation}
and 
\begin{equation}\label{3.2}
[\mathbb V(t)\phi](s)=\Gamma(s,s-t)\phi(s-t),\quad \forall\phi\in AP(\mathbb R,X_0).
\end{equation}

In the following, we investigate some properties of the one-parameter family $\{\mathbb U(t)\}_{t\ge0}$.
\begin{lem}\label{LE3.1}
Let Assumption \ref{ASS2.3} be satisfied. Then for any $\phi \in AP(\mathbb R,X_0)$, $[\mathbb U(t)\phi](s)$ is continuous in $s$.
\end{lem}
\begin{proof} Note that for any $\psi\in X_0$ and $t\ge s$, $U(t,s)\psi$ is a continuous function of $(t,s,\psi)$ (see Theorem \ref{TH2.6} (i)).  This implies that $\lim_{s_0\to 0}\|U(t+s_0,s+s_0)-U(t,s)\|=0$. In addition, for any $t\ge s$ and $\psi_1,\psi_2\in X_0$, we obtain that $\|U(t,s)\psi_1-U(t,s)\psi_2\|\to 0$, as $\|\psi_1-\psi_2\|\to 0$.

For any $s_0\in \mathbb R$, we have 
\begin{align*}
\|[\mathbb U(t)\phi](s_0+s)-[\mathbb U(t)\phi](s)\|&=\|U(s+s_0,s+s_0-t)\phi(s+s_0-t)-U(s,s-t)\phi(s-t)\|\\
&\le \|U(s+s_0,s+s_0-t)\phi(s+s_0-t)-U(s,s-t)\phi(s+s_0-t)\|\\
& \quad +\|U(s,s-t)\phi(s+s_0-t)-U(s,s-t)\phi(s-t)\|\\
&\le \|U(s+s_0,s+s_0-t)-U(s,s-t)\|\|\phi(s+s_0-t)\|\\
&\quad +\|U(s,s-t)\|\|\phi(s+s_0-t)-\phi(s-t)\|
\end{align*}
Based on the above, we can deduce
$\lim_{s_0\to0}\|[\mathbb U(t)\phi](s_0+s)-[\mathbb U(t)\phi](s)\|= 0.$
\end{proof}

\begin{lem}\label{LE3.2}
Let Assumption \ref{ASS2.3} be satisfied. Then $\{\mathbb U(t)\}_{t\ge0}$ is a semigroup on $AP(\mathbb R,X_0)$.
\end{lem}
\begin{proof}
Firstly, we prove that $\mathbb U(t)$ maps $AP(\mathbb R,X_0)$ to $AP(\mathbb R,X_0)$. For any $t\ge s$ and $\psi\in X_0$, we define $u(t,s,\psi)$ as the solution of \eqref{2.3} with initial value condition $u(s)=\psi$. Let $\phi\in AP(\mathbb R,X_0)$, thus we need to prove that $u(s,s-t,\phi(s-t))$ is almost periodic in $s$. For any $l\in \mathbb R$, we have
\begin{align}\label{3.3}
&\quad \| u(s+l,s-t+l,\phi(s-t+l))-u(s,s-t,\phi(s-t))\| \notag\\
&\le    \| u(s+l,s-t+l,\phi(s-t+l))-u(s,s-t,\phi(s-t+l))\| \\
&\quad +\| u(s,s-t,\phi(s-t+l))-u(s,s-t,\phi(s-t))\| \notag
\end{align}
Based on the Theorem \ref{TH2.6} (ii), we obtain that for $\epsi>0$ and $s\in \mathbb R$, there exists $\delta>0$ such that $\|F(s+l)-F(s)\|\le \delta$ and $\|V(s+l)-V(s)\|\le \delta$ for all $s\in \mathbb R$, then we have
\begin{equation}\label{3.4}
\| u(s+l,s-t+l,\phi(s-t+l))-u(s,s-t,\phi(s-t+l))\| \le \frac{\epsi}{2}.
\end{equation}
In addition, by Theorem \ref{TH2.6} (i), if $\|\phi(s-t+l)-\phi(s-t)\|\le \delta$ for $s-t\in \mathbb R$, then we have
\begin{equation}\label{3.5}
\| u(s,s-t,\phi(s-t+l))-u(s,s-t,\phi(s-t))\|\le \frac{\epsi}{2}. 
\end{equation}
Note that $F(t)$, $V(t)$ and $\phi(t)$ are almost periodic, then for any $\delta>0$, the sets $\mathcal T(F,\delta)$, $\mathcal T(V,\delta)$ and $\mathcal T(\phi,\delta)$ are relatively dense subsets of $\mathbb R$. By \cite[Theorem 1]{Fink69}, we obtain that the set $\mathcal T^*(\delta):=\mathcal T(F,\delta)\cap \mathcal T(V,\delta) \cap \mathcal T(\phi,\delta)$ is a relatively dense subset of $\mathbb R$. Let $u^*(s):=u(s,s-t,\phi(s-t))$ and $\mathcal T(u^*,\epsi) :=\{l\in\mathbb R:\| u^*(s+l)-u^*(s)\|\le \epsi,\forall s\in\mathbb R\}$. By \eqref{3.3}, \eqref{3.4} and \eqref{3.5}, we can deduce that
$\mathcal T^*(\delta)\subset \mathcal T(u^*,\epsi)$. This indicates that $\mathcal T(u^*,\epsi)$ is a relatively dense subset of $\mathbb R$. Thus, $u(s,s-t,\phi(s-t))$ is almost periodic in $s$.

In addition, it is easy to see that $\mathbb U(t_1+t_2)=\mathbb U(t_1)\mathbb U(t_2)$ for all $ t_1,t_2\ge0$. This implies that $\{\mathbb U(t)\}_{t\ge 0}$ is a semigroup. Therefore, $\{\mathbb U(t)\}_{t\ge0}$ is a semigroup on $AP(\mathbb R,X_0)$.
\end{proof}

\begin{lem}\label{LE3.3}
Let Assumption \ref{ASS2.3} be satisfied. Then $\mathbb U(t)$ is strongly continuous, i.e.,
$$
\lim_{t\downarrow 0}\|\mathbb U(t)\phi-\phi\|=0, \quad\forall \phi\in AP(\mathbb R,X_0).
$$
\end{lem}
\begin{proof}
Due to the exponential boundedness of the evolution family $\{U(t,s)\}_{t\ge s}$, we deduce that for any fixed $t$, there exists a constant $C_1>0$ such that $\| U(s,s-t)\|<C_1,\forall s\in \mathbb R$. In addition, by the fact that the range of an almost periodic function is a sequentially compact set, there exists a sequence $\{\varphi_i\}_{i=1}^n\in X_0$ such that $cls(R_{\phi(t)})\subset \mathop{\cup}_{i=1}^n B(\varphi_i,\frac{\epsi}{3C_1+1})$, where $R_{\phi(t)}$ is the range of the almost periodic function $\phi(t)$. Therefore, we can select a $i_t\in\{1,2,...,n\}$ such that
\begin{equation}\label{3.6}
\|\phi(t)-\varphi_{i_t}\|\le \frac{\epsi}{3C_1+1}.
\end{equation}
By Theorem \ref{TH2.6}, we know that $U(s,s-t)$ is a continuous function of  $s$ and $s-t$. Thus, for each $i\in \{1,2,...,n\}$, we can find $\delta_1>0$ such that 
\begin{equation}\label{3.7}
 \|\mathbb U(s,s-t)\varphi_i-\varphi_i\|\le \frac{\epsi}{3},\quad t\in(0,\delta_1).   
\end{equation}
Recall that almost periodic functions are uniformly continuous. Therefore, for any $\epsi>0$, there exists $\delta_2>0$ such that
\begin{equation}\label{3.8}
 \|\phi(s-t)-\phi(s)\|\le \frac{\epsi}{3}, \quad\forall t\in(0,\delta_2)   
\end{equation}
Let $\delta:=\min\{\delta_1,\delta_2\}$, it follows that for any $s\in\mathbb R$ and $t\in(0,\delta)$, we have
\begin{align*}
    \| [\mathbb U(t)\phi](s)-\phi(s)\|&=   \| U(s,s-t)\phi(s-t) -\phi(s)\| \\
&\le \|U(s,s-t)\phi(s-t) - U(s,s-t)\varphi_{i_{s-t}}   \|+\|U(s,s-t)\varphi_{i_{s-t}}  -\varphi_{i_{s-t}}  \|\\
&\quad +\|\varphi_{i_{s-t}}  -\phi(s-t)\|+\|\phi(s-t)-\phi(s)\| \\
&\le  (\|U(s,s-t)\|+1) \|  \|\varphi_{i_{s-t}}  -\phi(s-t)\|+\|U(s,s-t)\varphi_{i_{s-t}}  -\varphi_{i_{s-t}}  \|      \\
&\quad +  \|\phi(s-t)-\phi(s)\| 
\end{align*}
By \eqref{3.6}, \eqref{3.7} and \eqref{3.8}, we can deduce 
$
\lim_{t\downarrow 0}\|\mathbb U(t)\phi-\phi\|=0, \forall \phi\in AP(\mathbb R,X_0).
$
\end{proof}

By Lemmas \ref{LE3.1}, \ref{LE3.2} and \ref{LE3.3}, we can directly obtain the following theorem.
\begin{thm}\label{TH3.4}
Let Assumption \ref{ASS2.3} be satisfied. Then one-parameter family $\{\mathbb U(t)\}_{t\ge0}$ is a $C_0$-semigroup on Banach space $AP(\mathbb R,X_0)$.
\end{thm}

\begin{remk}\label{RE3.5}
Following the ideas of Lemmas \ref{LE3.1}, \ref{LE3.2}, \ref{LE3.3} and Theorem \ref{TH3.4}, we can deduce that $\{\mathbb V(t)\}_{t\ge 0}$ is also a $C_0$-semigroup on $AP(\mathbb R,X_0)$. By the proof above, we can find that the almost periodicity of $\{\mathbb U(t)\}_{t\ge 0}$ and $\{\mathbb V(t)\}_{t\ge 0}$ does not need Assumption \ref{ASS2.4}. In fact, the almost periodicity of functions $F(t)$ and $V(t)$ ensures the almost periodicity of the evolution semigroups $\{\mathbb U(t)\}_{t\ge 0}$ and $\{\mathbb V(t)\}_{t\ge 0}$. More precisely, if $H(t):X_0\to X$ is almost periodic with respect to $t$, then the following almost periodic non-densely defined Cauchy problem
$$\frac{du(t)}{dt}=Au(t)+H(t)u(t)$$
admits an evolution semigroup on $AP(\mathbb R, X_0)$.
\end{remk}
\begin{prop}\label{PR3.6}
Let Assumption \ref{ASS2.3} be satisfied. Then
$$
\omega(\mathbb U)=\omega(U),
$$
where $\omega(\cdot)$ represents the exponential growth bound.    
\end{prop}
\begin{proof}
Firstly, we define the exponential growth bound of the semigroup $\{\mathbb U(t)\}_{t\ge 0}$ as follows
$$
\omega(\mathbb U)=\inf\{\hat\omega\in \mathbb R:\exists M\ge 1:t\ge0:\|\mathbb U(t)\|\le Me^{\hat\omega t}\}.
$$
For each $f\in AP(\mathbb R,X_0)$, by the fact that
$$
\|[\mathbb U(t)f](t+s)\|=\|U(t+s,s)f(s)\|\le \sup_{s\in\mathbb R}\|U(t+s,s)\|\|f\|,
$$
it follows that $\|\mathbb U(t)\|\le \sup_{s\in\mathbb R}\|U(t+s,s)\|$.

Then we choose a sequence of almost periodic functions $\{g_i(t)\}_{i=1}^{\infty}\subset AP(\mathbb R, \mathbb R_+)$ such that $\|g_i\|=1,\forall i>1$ and $g_i(t)\to 1$ as $i\to +\infty$ uniformly on $\mathbb R$.  Let $w\in X_0$ with $\|w\|\le 1$ and $h_i(t)=g_i(t)w$. Then we have $\|h_i\|\le 1$ and 
$$
\|[\mathbb U(t)h_i](t+s)\|=\|U(t+s,s)h_i(s)\|=\|U(t+s,s)g_i(s)w\|.
$$
By Theorem \ref{TH2.6}, we know that $U(t+s,s)$ is continuous in $s$. Thus, we have
$$
\|U(t+s,s)w\|=\lim_{i\to+\infty}\|\mathbb U(t)h_i\|\le \lim_{i\to+\infty}\sup\|\mathbb U(t)\|\|h_i\|\le \|\mathbb U (t)\|
$$
Due to $w\in  X$ with $\|w\|\le 1$, we can deduce that $\sup_{s\in\mathbb R}\|U(t+s,s)\|\le \|\mathbb U(t)\|$. It follows from the definition of exponential growth bound that $\omega(\mathbb U)=\omega(U)$.
\end{proof}
\begin{remk}
Proposition \ref{PR3.6} is a classical result in the theory of the evolution semigroups. The idea of proof is from \cite[Theorem 3.23]{Chicone} and \cite[Lemma B.1]{Thieme09}. Similarly, we also can obtain $\omega(\mathbb V)=\omega(\Gamma)$.
\end{remk}

\subsection{Infinitesimal generators}
\indent\indent In this subsection, we investigate the infinitesimal generators of the evolution semigroups $\{\mathbb U(t)\}_{t\ge 0}$ and $\{\mathbb V(t)\}_{t\ge 0}$ respectively.

For $\lambda\in \mathbb R$ and $\lambda>\omega(\Gamma)$, we define an operator $\mathcal L_\lambda$ by
\begin{equation}\label{3.9}
\mathcal L_\lambda (\phi)(t):=\lim_{\mu\to +\infty}\int_{-\infty}^t e^{-\lambda(t-s)}\Gamma(t,s)\mu(\mu I-A)^{-1}\phi(s)ds, \quad t\in \mathbb R, \phi\in AP(\mathbb R,X).
\end{equation}
In addition, for $\lambda\in \mathbb R$, $\lambda>\omega(\Gamma)$ and $\mu\in \mathbb R$, $\mu>\max\{\omega({\Gamma}),\omega_A\}$, we define another operator $\mathbb L_\lambda^\mu$ by
\begin{equation}\label{3.10}
\mathbb L_\lambda^\mu (\phi)(t):=\int_{-\infty}^t e^{-\lambda(t-s)}\Gamma(t,s)\mu(\mu I-A)^{-1}\phi(s)ds,  \quad t\in \mathbb R, \phi\in AP(\mathbb R,X).
\end{equation}
\begin{remk}
It is clear that $\lim_{\mu\to +\infty}\mathbb L_\lambda^\mu=\mathcal L_\lambda$. In addition, by \cite[Theorem 1.11]{Magal21}, the limits in \eqref{3.9} and \eqref{3.10} are uniform with respect to $t$ in $\mathbb R$.
\end{remk}

\begin{lem}\label{LE3.9}\cite[Theorem 1.11]{Magal21}
Let Assumptions \ref{ASS2.3} and \ref{ASS2.4} be satisfied. Then the following statements are valid.
\begin{itemize}
\item[{\rm (i)}] For any $f\in BUC(\mathbb R,X),$ the Cauchy problem
$$
\frac{du(t)}{dt}=Au(t)-V(t)u(t)+f(t), \quad t\ge t_0,
$$  
admits a unique entire mild solution $u(t)\in C(\mathbb R,X_0)$ and $u(t)$ has the following expression
\begin{equation}\label{3.11}
    u(t)=\lim_{\mu\to+\infty}\int_{-\infty}^t \Gamma(t,s) \mu(\mu I-A)^{-1}f(s)ds,\quad t\in \mathbb R,
\end{equation}
where $BUC(\mathbb R,X)$ is the Banach space of bounded and uniformly continuous functions and the limit of \eqref{3.11} is uniform with respect to $t$ in $\mathbb R$. Moreover, if $f\in AP(\mathbb R,X_0)$, \eqref{3.11} admits the following expression
$$
u(t)=\int_{-\infty}^t \Gamma(t,s)f(s)ds.
$$
\item[{\rm (ii)}] For any $f\in BUC(\mathbb R,X)$, there exists a constant $C>0$ such that
$$
\|u\|\le C\|f\|.
$$
\end{itemize}
\end{lem}

Here, in order to obtain the almost periodicity of operator $\mathcal L_\lambda$, we add the following assumption.

\begin{ass}\label{ASS3.10}
For the evolution family $\{\Gamma(t,s)\}_{t\ge s}$, we assume that for any $\epsilon>0$ and $\nu>0$, there exist $\eta=\eta(\epsilon,\nu)>0$ and $\delta=\delta(\epsilon,\nu)<0$ such that
$$
\|\Gamma(t+\tau,s+\tau)-\Gamma(t,s)\|
\leq
\epsilon e^{\delta(t-s)}
$$
for all $t\geq s$ with $t-s\geq\nu$ and all $\tau\in \mathcal T(V,\eta)$, where $\mathcal T(V,\eta)$ is defined in \eqref{2.1}.
\end{ass}
\begin{remk}
Assumption \ref{ASS3.10} describes an almost periodic translation property of the evolution family $\{\Gamma(t,s)\}_{t\geq s}$. For nonautonomous parabolic evolution equations, estimates of this type can be obtained under suitable Acquistapace-Terreni conditions and almost periodicity assumptions on the associated operator family \cite[Proposition 4.4]{Maniar}. In fact, in many practical models, we can directly derive this assumption, as seen in \cite[Lemma 3.1]{Wang23}, \cite[Lemma 5.2]{Huo25} and the applications in this paper.
\end{remk}

\begin{thm}\label{TH3.12}
Let Assumptions \ref{ASS2.3}, \ref{ASS2.4} and \ref{ASS3.10} be satisfied. Then, for each $\lambda>\omega(\Gamma)$, $\mathcal L_\lambda$ is a positive bounded linear operator from $AP(\mathbb R,X)$ to $AP(\mathbb R,X_0)$.
\end{thm}
\begin{proof}
Recall that $A$ is a Hille-Yosida operator, then we have
$$\left\|\mu(\mu I-A)^{-1}\right\|\le \frac{\mu M_A}{\mu-\omega_A},\quad\forall\mu> \max\{0,\omega_A\}.$$ 
In addition, by Assumption \ref{ASS2.4}, we have $\|\Gamma(t,s)\|\le M_\Gamma e^{\omega(\Gamma)(t-s)}$ with $\omega(\Gamma)<0$. Thus, for each $\lambda\in \mathbb R$, $\lambda>\omega(\Gamma)$ and each $\mu\in \mathbb R$, $\mu>\max\{0,\omega_A\}$, we have

\begin{equation}\label{3.13}
    \|\mathbb L^\mu_\lambda(\phi)\|\le \frac{M_\Gamma}{\lambda-\omega(\Gamma)} \frac{\mu M_A}{\mu-\omega_A} \|\phi\|,\quad \forall \phi\in AP(\mathbb R,X).
\end{equation}
By setting $\mu\to +\infty$, we obtain
\begin{equation}\label{3.14}
    \|\mathcal L_\lambda(\phi)\|\le \frac{M_A M_\Gamma}{\lambda-\omega(\Gamma)} \|\phi\|.
\end{equation}
This implies that $\mathcal L_\lambda$ is a bounded operator for each $\lambda>\omega(\Gamma)$. Moreover, it follows from Lemma \ref{LE3.9} that $\mathcal L_\lambda$ maps $AP(\mathbb R, X)$ to $BUC(\mathbb R, X_0)$. In addition, by \eqref{3.9}, it is easy to check that $\mathcal L_\lambda$ is a positive operator.

Finally, we show that $\mathcal L_\lambda$ maps $AP(\mathbb R,X)$ to $AP(\mathbb R,X_0)$.  Let $\phi\in AP(\mathbb R,X)$ and
$\epsilon>0$ be given. Fix
$\mu>\max\{0,\omega_A\}$. Since
$\lambda>\omega(\Gamma)$, we can choose $\nu>0$ sufficiently small
and $H>\nu$ sufficiently large such that
$$
2M_\Gamma\frac{\mu M_A}{\mu-\omega_A}\|\phi\|
\int_{t-\nu}^{t}
e^{(\omega(\Gamma)-\lambda)(t-s)}ds
<\frac{\epsilon}{4}
$$
and
$$
2M_\Gamma\frac{\mu M_A}{\mu-\omega_A}\|\phi\|
\int_{-\infty}^{t-H}
e^{(\omega(\Gamma)-\lambda)(t-s)}ds
<\frac{\epsilon}{4}.
$$
Notice that the above two integrals are independent of $t$. In addition, we take $\epsilon_1>0$ sufficiently small such that
$$
\frac{\mu M_A}{\mu-\omega_A}\|\phi\|\epsilon_1
\int_{t-H}^{t-\nu}
e^{-\lambda(t-s)}ds
<\frac{\epsilon}{4}.
$$
By Assumption \ref{ASS3.10}, there exist $\eta=\eta(\epsilon_1,\nu)>0$ and $\delta=\delta(\epsilon_1,\nu)<0$ such that
$$
\|\Gamma(t+\tau,s+\tau)-\Gamma(t,s)\| \le\epsilon_1e^{\delta(t-s)}
$$
for all $t-s\geq\nu$ and $\tau\in\mathcal T(V,\eta)$. Furthermore, we take $\epsilon_2>0$ sufficiently small such that
$$
\frac{M_\Gamma}{\lambda-\omega(\Gamma)}
\frac{\mu M_A}{\mu-\omega_A}\epsilon_2
<\frac{\epsilon}{4}.
$$
Since $V(t)$ is an almost periodic function and $\phi\in AP(\mathbb R,X)$,
the function $t\mapsto (V(t),\phi(t))$ is almost periodic. Hence, there exists a relatively dense set $\mathcal T_*\subset\mathbb R$ such that, for any $\tau\in\mathcal T_*$, 
$$
\tau\in \mathcal T(V,\eta) \text{ and }\|\phi(t+\tau)-\phi(t)\|<\epsilon_2,
\quad t\in\mathbb R,
$$
where $\eta>0$ is given by Assumption \ref{ASS3.10}.

For any $\tau\in\mathcal T_*$, we have

$$
\begin{aligned}
&\quad \|\mathbb L_\lambda^\mu(\phi)(t+\tau)-\mathbb L_\lambda^\mu(\phi)(t)\| \\ 
&=  \left\| \int_{-\infty}^{t+\tau} e^{-\lambda(t+\tau-s)}\Gamma(t+\tau,s)\mu(\mu I-A)^{-1}\phi(s)ds-\int_{-\infty}^t e^{-\lambda(t-s)}\Gamma(t,s)\mu(\mu I-A)^{-1}\phi(s)ds\right\| \\
&= \left\| \int_{-\infty}^{t} e^{-\lambda(t-s)}\Gamma(t+\tau,s+\tau)\mu(\mu I-A)^{-1}\phi(s+\tau)ds-\int_{-\infty}^t e^{-\lambda(t-s)}\Gamma(t,s)\mu(\mu I-A)^{-1}\phi(s)ds\right\| \\
&\le \sup_{t\in \mathbb R}\left[ \int_{-\infty}^{t} e^{-\lambda(t-s)}\left\| \Gamma(t+\tau,s+\tau)-\Gamma(t,s) \right\| \left\| \mu(\mu I-A)^{-1}\right\| \left\|\phi(s+\tau) \right\|ds\right.\\
&\quad\left.+\int_{-\infty}^{t} e^{-\lambda(t-s)}\left\| \Gamma(t,s) \right\| \left\| \mu(\mu I-A)^{-1}\right\| \left\|\phi(s+\tau)-\phi(s) \right\|  ds  \right]\\
&\le
\sup_{t\in \mathbb R}\left\{\frac{\mu M_A}{\mu-\omega_A}\|\phi\|
\left[
\int_{t-\nu}^{t}
e^{-\lambda(t-s)}
\|\Gamma(t+\tau,s+\tau)-\Gamma(t,s)\|ds\right.\right.\left.+
\int_{t-H}^{t-\nu}
e^{-\lambda(t-s)}
\|\Gamma(t+\tau,s+\tau)-\Gamma(t,s)\|ds\right.\\
&\quad+\left.\left.\int_{-\infty}^{t-H}e^{-\lambda(t-s)}\|\Gamma(t+\tau,s+\tau)-\Gamma(t,s)\|ds
\right] +
M_\Gamma\frac{\mu M_A}{\mu-\omega_A}\epsilon_2
\int_{-\infty}^{t}
e^{(\omega(\Gamma)-\lambda)(t-s)}ds \right\}\\
&\le
2M_\Gamma\frac{\mu M_A}{\mu-\omega_A}\|\phi\|
\int_{t-\nu}^{t}
e^{(\omega(\Gamma)-\lambda)(t-s)}ds+\frac{\mu M_A}{\mu-\omega_A}\|\phi\|\epsilon_1\int_{t-H}^{t-\nu}e^{-\lambda(t-s)}ds\\
&\quad+
2M_\Gamma\frac{\mu M_A}{\mu-\omega_A}\|\phi\|
\int_{-\infty}^{t-H}
e^{(\omega(\Gamma)-\lambda)(t-s)}ds+\sup_{t\in \mathbb R}\left\{M_\Gamma\frac{\mu M_A}{\mu-\omega_A}\epsilon_2\int_{-\infty}^{t}e^{(\omega(\Gamma)-\lambda)(t-s)}ds\right\}\\
&<\epsilon.
\end{aligned}
$$
Since $\mathcal T_*$ is relatively dense, we obtain
$\mathbb L_\lambda^\mu(\phi)\in AP(\mathbb R,X_0).$
By letting $\mu\to+\infty$ and using the uniform convergence of
$\mathbb L_\lambda^\mu(\phi)$ to $\mathcal L_\lambda(\phi)$, we conclude that $\mathcal L_\lambda(\phi)\in AP(\mathbb R,X_0).$
\end{proof}

Next, inspired by Djidjou-Demasse et al.\cite{Djidjou-Demasse}, we give two theorems to relate the infinitesimal generator of the evolution semigroup $\{\mathbb V(t)\}_{t\ge 0}$ to the operator $\mathcal L_\lambda$. Here, we extend Djidjou-Demasse et al.'s results on the periodic case to the almost periodic case.
\begin{thm}\label{TH3.13}
Let Assumptions \ref{ASS2.3}, \ref{ASS2.4} and \ref{ASS3.10} be satisfied and $\mathcal B_0:D(\mathcal B_0)\subset AP(\mathbb R,X_0)\to AP(\mathbb R,X_0)$ be the infinitesimal generator of the evolution semigroup $\{\mathbb V(t)\}_{t\ge 0}$.  Then the following statements are valid.
\begin{itemize}
\item[{\rm (i)}] There exists a Hille-Yosida operator $\mathcal B:D(\mathcal B)\subset AP(\mathbb R,X)\to AP(\mathbb R,X)$ such that $\mathcal B_0$ is the part of $\mathcal B$ in $AP(\mathbb R,X_0)$.
\item[{\rm (ii)}]  $(\omega(\Gamma),+\infty)\subset \rho(\mathcal B)$, where $\rho(\mathcal B)$ is the resolvent set of operator $\mathcal B$.
\item[{\rm (iii)}] For any $\phi\in AP(\mathbb R,X)$ and $\lambda>\omega(\Gamma)$, $(\lambda I-\mathcal B)^{-1}(\phi)=\mathcal L_\lambda(\phi)$ and $(\lambda I-\mathcal B)^{-n}(\phi)$ has the following estimate
\begin{equation}\label{3.15}
    \left\| (\lambda I-\mathcal B)^{-n}(\phi)\right\| \le \frac{M_\Gamma^2 M_A}{(\lambda-\omega(\Gamma))^{n}}\|\phi\|,\quad \forall n\ge 1.
\end{equation}
\end{itemize}
\end{thm}
\begin{proof}
By replacing $\lambda$ by $\lambda_1$ and $\lambda_2$ in \eqref{3.9}, we can deduce that
\begin{equation}\label{3.16}
    \mathcal L_{\lambda_1}\circ \mathcal L_{\lambda_2}(\phi)(t)=\int_{-\infty}^{t} e^{-\lambda_1(t-s)}\Gamma(t,s)\mathcal L_{\lambda_2}(\phi)(s)ds,\quad \forall \lambda_1,\lambda_2>\omega(\Gamma),t\in \mathbb R.
\end{equation}

We begin with the following Claim.

\textbf{Claim:} $\mathcal L_\lambda^n(\phi)(t)=\int_{-\infty}^t \frac{(t-s)^{n-2}}{(n-2)!} e^{-\lambda(t-s)}\Gamma(t,s) \mathcal L_\lambda (\phi)(s)ds,\forall n\ge 2, \phi\in AP(\mathbb R,X),t\in \mathbb R$.

We prove this \textbf{Claim} by induction. By setting $\lambda_1=\lambda_2$ in \eqref{3.16}, we can deduce that the \textbf{Claim} holds for $n=2$. Assume the \textbf{Claim} is true for some $n\ge 2$. Thus, for each $n\ge 2$, we have
$$
\begin{array}{*{20}{rl}}
\mathcal L_\lambda^{n+1}(\phi)(t)&=\int_{-\infty}^t e^{-\lambda(t-s)}\Gamma(t,s) \mathcal L_\lambda^n(\phi)(s)ds\\
&=\int_{-\infty}^t \int_{-\infty}^{s} \frac{(s-l)^{n-2}}{(n-2)!} e^{-\lambda(t-l)}\Gamma(t,l) \mathcal L_{\lambda} (\phi)(l)dlds\\
&=\int_{-\infty}^t \int_l^t \frac{(s-l)^{n-2}}{(n-2)!} e^{-\lambda(t-l)} \Gamma(t,l) \mathcal L_\lambda(\phi)(l)dsdl\\
&=\int_{-\infty}^t \frac{(t-l)^{n-1}}{(n-1)!} e^{-\lambda(t-l)} \Gamma(t,l) \mathcal L_{\lambda}(\phi)(l)dl.
\end{array}
$$
This implies that the \textbf{Claim} is true. Therefore, by using \eqref{3.14}, we can see
$$
\| \mathcal L_\lambda^n(\phi)\|\le \frac{M_\Gamma}{(\lambda-\omega(\Gamma))^{n-1}}\|\mathcal L_\lambda(\phi)\| \le \frac{M_A M_\Gamma^2}{(\lambda-\omega(\Gamma))^n}\|\phi\|,\quad \forall \lambda>\omega(\Gamma).
$$

Next, we show that $\mathcal L_\lambda$ is a pseudo-resolvent, that is, $\mathcal L_\lambda$ satisfies the resolvent equation
\begin{equation}\label{3.17}
    \mathcal L_{\lambda_1}- \mathcal L_{\lambda_2}=(\lambda_2-\lambda_1) \mathcal L_{\lambda_1} \mathcal L_{\lambda_2}, \quad \forall \lambda_1,\lambda_2>\omega(\Gamma).
\end{equation}
By \eqref{3.16}, we have
\begin{equation}\label{3.18}
    \mathcal L_{\lambda_1}\circ \mathcal L_{\lambda_2}(\phi)(t)=\int_{-\infty}^t e^{-\lambda_1(t-s)}\Gamma(t,s) \lim_{\mu\to \infty}\left( \int_{-\infty}^s e^{-\lambda_2(s-l)} \Gamma(s,l) \mu(\mu I-A)^{-1}\phi(l)dl  \right)ds.
\end{equation}
It follows from Lemma \ref{LE3.9} that the limit in \eqref{3.18} exists uniformly for $t\in \mathbb R$. Then we have
$$
\begin{array}{*{20}{rl}}
\mathcal L_{\lambda_1}\circ \mathcal L_{\lambda_2}(\phi)(t)&=\lim_{\mu\to+\infty}\int_{-\infty}^t \int_{-\infty}^s e^{-\lambda_2(s-l)} e^{-\lambda_1(t-s)}\Gamma(t,l)\mu(\mu I-A)^{-1}\phi(l)dlds\\
&=\lim_{\mu\to +\infty}\int_{-\infty}^t \int_{l}^t e^{-\lambda_2(s-l)} e^{-\lambda_1(t-s)}\Gamma(t,l)\mu(\mu I-A)^{-1}\phi(l)dsdl\\
&=\lim_{\mu\to\infty}\int_{-\infty}^t \frac{e^{-\lambda_2(t-l)}-e^{-\lambda_1(t-l)}}{\lambda_1-\lambda_2} \Gamma(t,l)\mu(\mu I-A)^{-1}\phi(l)dl  \\
&=\frac{1}{\lambda_1-\lambda_2}(\mathcal L_{\lambda_2}(\phi)(t)-\mathcal L_{\lambda_1}(\phi)(t)).
\end{array}
$$
This implies that $\mathcal L_\lambda$ is a pseudo-resolvent. In addition, by Lemma \ref{LE3.9} (ii), we can see
$\mathcal L_\lambda(\phi)=0 \Longleftrightarrow \phi=0$. It follows from \cite[Proposition B.6]{Arendt} that there exists a closed linear operator $\mathcal B$ such that $\mathcal L_\lambda=(\lambda I-\mathcal B)^{-1}$. Thus, statements (ii) and (iii) are true.

Finally, we show statement (i). For any $\phi\in AP(\mathbb R,X_0)$ and $t\in \mathbb R$, we have
$$
\begin{array}{*{20}{rl}}
(\lambda I-\mathcal B)^{-1}(\phi)(t)&= \lim_{\mu\to \infty}\int_{-\infty}^t e^{-\lambda(t-s)}\Gamma(t,s) \mu(\mu I-A)^{-1}\phi(s)ds\\
&=\int_{-\infty}^t e^{-\lambda(t-s)}\Gamma(t,s) \phi(s)ds\\
&=\int_0^\infty e^{-\lambda s} [\mathbb V(s)\phi](t)ds.
\end{array}
$$
Note that $\mathcal B_0$ is the infinitesimal generator of the evolution semigroup $\{\mathbb V(t)\}_{t\ge 0}$. By the integral representation of the resolvent of the evolution semigroup $\{\mathbb V(t)\}_{t\ge 0}$ \cite[Section 1]{Engel}, then we have 
\begin{equation}\label{3.19}
    (\lambda I-\mathcal B)^{-1}(\phi)(t)=\int_0^\infty e^{-\lambda s} [\mathbb V(s)\phi](t)ds=(\lambda I-\mathcal B_0)^{-1}(\phi)(t).
\end{equation}
Then we can deduce $D(\mathcal B_0)\subset D(\mathcal B)$ and $\mathcal B_0(\phi)=\mathcal B(\phi),\forall \phi\in AP(\mathbb R,X_0)$. By Theorem \ref{TH3.12} and $\mathcal L_\lambda=(\lambda I-\mathcal B)^{-1}$, we can obtain that $D(\mathcal B)\subset AP(\mathbb R,X_0)$. Let $\phi\in D(\mathcal B)$ such that $\mathcal B\phi \in AP(\mathbb R,X_0)$. It follows from \eqref{3.19} that
$$
\phi=(\lambda I-\mathcal B)^{-1}(\lambda I-\mathcal B)(\phi)=\lambda(\lambda I-\mathcal B_0)^{-1}(\phi)-(\lambda I-\mathcal B_0)^{-1}\mathcal B(\phi)\in D(\mathcal B_0).
$$
Combining with the definition given by \eqref{2.4}, statement (i) holds.
\end{proof}

Recall that $F(t)\in \mathcal L(X_0,X)$ is positive and almost periodic with respect to $t$, $\mathcal B+F$ can be seen as a bounded positive perturbation of $\mathcal B$. By bounded perturbation Theorem of $C_0-$semigroup (see \cite[Theorem 3.1.3]{Engel}), the part $(\mathcal B+F)_0$ of $\mathcal B+F$ generates a $C_0-$semigroup $\{T_{{(\mathcal B+F)}_0}(t)\}_{t\ge0}$ on $AP(\mathbb R,X_0)$. In addition, by \cite[Corollary 3.1.7]{Engel}, $\{T_{{(\mathcal B+F)}_0}(t)\}_{t\ge0}$ has the following uniquely determined form
\begin{equation}\label{3.20}
    [T_{{(\mathcal B+F)}_0}(t)(\phi)]=[\mathbb V(t)\phi]+\lim_{\lambda\to \infty}\int_0^t [\mathbb V(t-l) \lambda(\lambda I-\mathcal B)^{-1} [FT_{{(\mathcal B+F)}_0}(l)\phi]]dl, \forall t\ge 0,\phi\in AP(\mathbb R,X_0),
\end{equation}
where $[FT_{(\mathcal B+F)_0}(l)\phi](s)=F(s)[T_{(\mathcal B +F)_0}(l)\phi](s),\forall s\in \mathbb R$. In \eqref{3.20}, $(s)$ in $[FT_{(\mathcal B+F)_0}(l)\phi](s)$ is determined by the evolution semigroup $\{\mathbb V(t)\}_{t\ge 0}$, so we omit it.

\begin{thm}\label{TH3.14}
Let Assumptions \ref{ASS2.3}, \ref{ASS2.4} and \ref{ASS3.10} be satisfied. Then strongly continuous semigroups $\{T_{{(\mathcal B+F)}_0}(t)\}_{t\ge0}$ and $\{\mathbb U(t)\}_{t\ge 0}$ coincide in $AP(\mathbb R,X_0)$, that is, for any $t\ge 0$,$s\in \mathbb R$ and $\phi\in AP(\mathbb R,X_0)$, the following expression holds

\begin{equation}\label{3.21}
    [\mathbb U(t)\phi](s)=[T_{{(\mathcal B+F)}_0}(t)\phi](s).
\end{equation}
\end{thm}
\begin{proof}
By Theorem \ref{TH2.6}, we have 
$$
U(s,s-t)x =\Gamma(s,s-t)x+\lim_{\mu\to \infty} \int_{s-t}^s \Gamma (s,l)\mu(\mu I-A)^{-1} F(l) U(l,s-t)xdl,\quad \forall s\in \mathbb R,t\ge 0,x\in X_0.
$$
By using the evolution semigroups $\{\mathbb U(t)\}_{t\ge 0}$ and $\{\mathbb V(t)\}_{t\ge 0}$, we have, for any $\phi\in AP(\mathbb R,X_0)$,    
\begin{equation}\label{3.22}
    \begin{array}{*{20}{rl}}
&\quad[\mathbb U(t)\phi] (s)\\
&=[\mathbb V(t)\phi](s)+\lim_{\mu\to \infty}\int_{s-t}^s \Gamma(s,l)\mu (\mu I-A)^{-1} F(l) U(l,s-t)\phi(s-t)dl\\
& =[\mathbb V(t)\phi](s)+\mathop{\lim}_{\mu\to \infty}\int_0^t \Gamma(s,s-t+l)\mu (\mu I-A)^{-1} F(s-t+l) U(s-t+l,s-t)\phi(s-t)dl\\
&=[\mathbb V(t)\phi](s)+\mathop{\lim}_{\mu\to \infty}\int_0^t \Gamma(s,s-t+l)\mu (\mu I-A)^{-1} F(s-t+l) [\mathbb U(l)\phi](s-t+l)dl.
\end{array}
\end{equation}
Let $g(t,s)$ be defined as follows
\begin{equation}\label{3.23}
g(t,s):=\mathop{\lim}_{\mu\to \infty}\int_0^t \Gamma(s,s-t+l)\mu (\mu I-A)^{-1} F(s-t+l) [\mathbb U(l)\phi](s-t+l)dl.
\end{equation}
Due to $\phi\in AP(\mathbb R,X_0)$, we have $g(t,s)\in X_0$. For any $\lambda >\omega(\Gamma)$,by Theorem \ref{TH3.13}, we have
$$
h(t,s):=(\lambda I-\mathcal B)^{-1}(g(t,\cdot))(s)=\int_{-\infty}^s e^{-\lambda(s-r)} \Gamma(s,r)g(t,r)dr,\quad \forall s\in\mathbb R.
$$
Then we have
$$
\begin{array}{*{20}{rl}}
h(t,s)&=\lim_{\mu\to \infty}\int_{-\infty}^{s}\int_0^t e^{-\lambda(s-r)}\Gamma(s,r-t+l) \mu (\mu I-A)^{-1} F(r-t+l) [\mathbb U(l) \phi](r-t+l)dldr  \\
&=\lim_{\mu\to \infty}\int_0^t\int_{-\infty}^{s} e^{-\lambda(s-r)}\Gamma(s,r-t+l) \mu (\mu I-A)^{-1} F(r-t+l) [\mathbb U(l) \phi](r-t+l)drdl\\
&= \lim_{\mu\to \infty} \int_0^t \int_{-\infty}^{s-t+l} e^{-\lambda(s-t+l-r)} \Gamma(s,r) \mu (\mu I-A)^{-1} F(r) [\mathbb U(l)\phi](r)drdl\\
&= \lim_{\mu\to \infty} \int_0^t \Gamma(s,s-t+l) \int_{-\infty}^{s-t+l} e^{-\lambda(s-t+l-r)} \Gamma(s-t+l,r) \mu(\mu I-A)^{-1} F(r)[\mathbb U(l)\phi](r)drdl\\
&=\int_0^t \Gamma(s,s-t+l) (\lambda I-\mathcal B)^{-1}F(s-t+l) [\mathbb U(l)\phi](s-t+l)dl\\
&=\int_0^t [\mathbb V(t-l)(\lambda I-\mathcal B)^{-1} [F\mathbb U(l)\phi]](s)dl.
\end{array}
$$
By the definition of $g$ and $h$, we can see
$$
\lim_{\lambda\to+\infty} \lambda h(t,s)=\lim_{\lambda\to +\infty}\lambda(\lambda I-\mathcal B)^{-1}(g(t,\cdot))(s).
$$
Since $g(t,s)\in X_0$, we have 
$$
g(t,s)=\lim_{\lambda\to +\infty}\lambda(\lambda I-\mathcal B)^{-1}(g(t,\cdot))(s)=\lim_{\lambda\to +\infty}\int_0^t [\mathbb V(t-l)\lambda(\lambda I-\mathcal B)^{-1} [F\mathbb U(l)\phi]](s)dl.
$$
By \eqref{3.22} and \eqref{3.23}, we have 
\begin{equation}\label{3.24}
    [\mathbb U(t)\phi] (s)=[\mathbb V(t)\phi](s)+\lim_{\lambda\to +\infty}\int_0^t [\mathbb V(t-l)\lambda(\lambda I-\mathcal B)^{-1} [F\mathbb U(l)\phi]](s)dl.
\end{equation}
Recall that \eqref{3.20} is the uniquely determined expression of semigroup $\{T_{{(\mathcal B+F)}_0}(t)\}_{t\ge0}$. In addition, we can find that \eqref{3.20} has the same form as \eqref{3.24}. This implies that  semigroups $\{T_{{(\mathcal B+F)}_0}(t)\}_{t\ge0}$ and $\{\mathbb U(t)\}_{t\ge 0}$ coincide in $AP(\mathbb R,X_0)$.
\end{proof}
\begin{remk}\label{RE3.15}
Since the strongly continuous semigroups $\{T_{{(\mathcal B+F)}_0}(t)\}_{t\ge0}$ and $\{\mathbb U(t)\}_{t\ge 0}$ coincide in $AP(\mathbb R,X_0)$, we can use the part $(\mathcal B+F)_0$ of $\mathcal B+F$ as the  infinitesimal generator of the evolution semigroup $\{\mathbb U(t)\}_{t\ge 0}$. In fact, if we consider a specific model, the above infinitesimal generators can be directly calculated without using the above theorem, such as the almost periodic population model with age structure \cite{Huo25}.
\end{remk}

\subsection{Basic reproduction ratio}
\indent\indent In this subsection, we establish the theory of the basic reproduction ratio $R_0$ for the almost periodic non-densely defined Cauchy problem \eqref{2.3}. 

We define an integral operator $\mathscr L$ on $AP(\mathbb R,X)$ by
$$
\mathscr L(\phi)(t):=F(t)\lim_{\mu\to +\infty} \int_{-\infty}^t \Gamma(t,s)\mu(\mu I-A)^{-1}\phi(s)ds,\quad \forall t\in \mathbb R,\phi\in AP(\mathbb R,X).
$$
\begin{lem}\label{LE3.16}
Let Assumptions \ref{ASS2.3}, \ref{ASS2.4} and \ref{ASS3.10} be satisfied. Then $\mathscr L$ is a linear positive operator from $AP(\mathbb R,X)$ to $AP(\mathbb R,X)$.
\end{lem}
\begin{proof}
Since $F(t)\in \mathcal L(X_0,X)$ is almost periodic, $F(t)\phi(t)$ is also almost periodic on $X$. By a similar proof as Theorem \ref{TH3.12}, we can deduce that $\mathscr L$ is a linear positive operator from $AP(\mathbb R,X)$ to $AP(\mathbb R,X)$.
\end{proof}

{ In the theory of population dynamics, $\mathscr L$ is the \textbf{next generation operator} of the almost periodic non-densely defined Cauchy problem \eqref{2.3}. Therefore, we can define the basic reproduction ratio by
$$
R_0=r(\mathscr L),
$$
where $r(\mathscr L)$ denotes the spectral radius of $\mathscr L$.
}

\begin{thm}\label{TH3.17}
Let Assumptions \ref{ASS2.3}, \ref{ASS2.4} and \ref{ASS3.10} be satisfied. Then $R_0-1$ has the same sign as $\omega(\mathbb U)$, where $\omega(\mathbb U)$ is the exponential growth bound of the evolution semigroup $\{\mathbb U(t)\}_{t\ge 0}$.
\end{thm}
\begin{proof}
Note that Banach space $X_0$ admits a normal, generating and positive cone $X^+_0$,  then we can obtain that $AP(\mathbb R,X_0)$ is also a Banach space with a normal, generating and positive cone $AP(\mathbb R,X^+_0)$. By Theorem \ref{TH3.13}, we know that $\mathcal B_0$ is the infinitesimal generator of the evolution semigroup $\{\mathbb V(t)\}_{t\ge0}$. In addition, it follows from Theorem \ref{TH3.14} that $(\mathcal B+F)_0$ is the infinitesimal generator of the evolution semigroup $\{\mathbb U(t)\}_{t\ge0}$. By Theorem \ref{TH3.4} and Remark \ref{RE3.5}, we know that $\{\mathbb U(t)\}_{t\ge0}$ and $\{\mathbb V(t)\}_{t\ge0}$ are positive $C_0-$semigroups on $AP(\mathbb R,X_0)$. According to \cite[Theorem 3.12]{Thieme09}, $\mathcal B_0$ and $(\mathcal B+F)_0$ are resolvent positive operators. By using the spectral mapping theorem for the almost periodic evolution semigroups in Banach space \cite[Theorem 3.6]{Hutter}, we have 
\begin{equation}\label{3.25}
    \sigma(\mathbb U(t))\setminus\{0\}=e^{\sigma((\mathcal B+F)_0)t}\text{ and  } \sigma(\mathbb V(t))\setminus\{0\}=e^{\sigma(\mathcal B_0)t}.
\end{equation}
where $\sigma(\cdot)$ denotes the spectrum. Then we have $r(\mathbb U(t))=e^{s((\mathcal B+F)_0)t}$ and $r(\mathbb V(t))=e^{s(\mathcal B_0)t}$, where $s(\cdot)$ denotes the spectral bound.  Based on the property of $C_0-$semigroups \cite[Proposition 6.2.2]{Engel}, we know that 
\begin{equation}\label{3.26}
    r(\mathbb U(t))=e^{\omega(\mathbb U)t}\text{ and  } r(\mathbb V(t))=e^{\omega(\mathbb V)t},
\end{equation}
Thus, we can find $s((\mathcal B+F)_0)=\omega(\mathbb U)$ and $s(\mathcal B_0)=\omega(\mathbb V)<0$.  According to \cite[Lemma 2.1 and Lemma 2.2]{Magal09MAMS}, we can deduce that $\sigma((\mathcal B+F)_0)=\sigma(\mathcal B+F)$ and  $\sigma(\mathcal B_0)=\sigma(\mathcal B)$. Thus,  $s(\mathcal B+F)=s((\mathcal B+F)_0)=\omega(\mathbb U)$ and $s(\mathcal B)=s(\mathcal B_0)=\omega(\mathbb V)<0$. 

Note that $F$ is positive and $\mathcal B$ is a resolvent positive operator with $s(\mathcal B)<0$, it follows from \cite[Lemma 5.8]{Thieme09} that $r(-F\mathcal B^{-1})-1$ has the same sign as $s(\mathcal B+F)$. By Theorem \ref{TH3.13} and setting $\lambda=0$, we can obtain
$$
(-\mathcal B^{-1}\phi)(t)=  \lim_{\mu\to +\infty}\int_{-\infty}^t \Gamma(t,s) \mu(\mu I-A)^{-1}\phi(s)ds,\quad t\in\mathbb R, \phi\in AP(\mathbb R,X).
$$
Then $-F\mathcal B^{-1}$ has the following form
\begin{equation}\label{3.27}
(-F\mathcal B^{-1}\phi)(t)= F(t) \lim_{\mu\to +\infty}\int_{-\infty}^t \Gamma(t,s) \mu(\mu I-A)^{-1}\phi(s)ds,\quad t\in\mathbb R, \phi\in AP(\mathbb R,X).
\end{equation}
It is easy to see that $\mathscr L=-F\mathcal B^{-1}$. Since $s(\mathcal B+F)=\omega(\mathbb U)$ and $r(-F\mathcal B^{-1})-1$ has the same sign as $s(\mathcal B+F)$,  we can deduce that $R_0-1$ has the same sign as $\omega(\mathbb U)$.
\end{proof}

{
\begin{remk}
Under Assumptions \ref{ASS2.3}, \ref{ASS2.4} and \ref{ASS3.10}, we can see that $\mathscr L$ is a positive operator and hence $r(\mathscr L)\ge 0$. In general, the strict positivity of $r(\mathscr L)$ is not guaranteed under the above assumptions, and it is not required for Theorem \ref{TH3.17}. The strict positivity of $R_0$ can be obtained under stronger
positivity assumptions. For example, if we additionally assume
\begin{itemize}
\item[{\rm (i)}] $X_+$ and $X_0^+$ have nonempty interiors.
\item[{\rm (ii)}] $-\mathcal B^{-1}$ is strongly positive from $AP(\mathbb R,X)$ to $AP(\mathbb R,X_0)$.
\item[{\rm (iii)}] $F(t)$ is strongly positive from $AP(\mathbb R,X_0)$ to
$AP(\mathbb R,X)$.
\end{itemize}
Then the next generation operator $\mathscr L$ is strongly positive and hence $r(\mathscr L)>0$. In concrete applications, the strong positivity of $-\mathcal B^{-1}$ can often be verified from the integral representation of $-\mathcal B^{-1}$ and suitable positive properties of the evolution family
$\{\Gamma(t,s)\}_{t\geq s}$.
\end{remk}
}

By Theorem \ref{TH3.17}, we can directly obtain the following corollary.
\begin{coro}\label{CO3.18}
Let Assumptions \ref{ASS2.3}, \ref{ASS2.4} and \ref{ASS3.10} be satisfied. Then the following statements are valid.
\begin{itemize}
\item[{\rm (i)}] If $R_0<1$, then $\omega(\mathbb U)<0$.
\item[{\rm (ii)}]  If $R_0=1$, then $\omega(\mathbb U)=0$.
\item[{\rm (iii)}]  If $R_0>1$, then $\omega(\mathbb U)>0$.
\end{itemize}    
\end{coro}
\begin{remk}\label{RE3.19}
In the proof of Theorem \ref{TH3.17}, we can see that $R_0$ is equal to $r(-F\mathcal B^{-1})$, where $F$ and $\mathcal B$ are non-densely defined operators. This implies that $R_0$ of the non-densely defined Cauchy problems can be defined as
the same form as ODE, reaction-diffusion systems and nonlocal diffusion systems.
It is worth mentioning that Corollary \ref{CO3.18} is an important property and satisfied in many epidemic models, for example, time delay epidemic models\cite{Zhao17,Liang19}, reaction-diffusion epidemic models\cite{Wang12,Zhang21} and almost periodic epidemic models\cite{Qiang,Wang23}.
\end{remk}

{

For each $\gamma>0$, we consider the following linear almost periodic non-densely defined Cauchy problem
\begin{equation}\label{3.28}
\frac{du(t)}{dt}=Au(t)+\frac{1}{\gamma}F(t)u(t)-V(t)u(t).
\end{equation}
Since $\frac{1}{\gamma}F(t)$ satisfies the same assumptions as $F(t)$, by a similar argument as that for Cauchy problem \eqref{2.3}, we can deduce that  Cauchy problem \eqref{3.28} admits an evolution family $\{U_\gamma(t,s)\}_{t\ge s}$ on $X_0$.

\begin{lem}\label{LEM-R0-computation}
Let Assumptions \ref{ASS2.3}, \ref{ASS2.4} and \ref{ASS3.10}
be satisfied, and assume that $R_0>0$. 
Then $\gamma=R_0$ is the unique positive solution of
$$
\omega(U_\gamma)=0,
$$
where $\omega(U_\gamma)$ is the exponential growth bound of $\{U_\gamma(t,s)\}_{t\ge s}$. More precisely,
\[
\begin{cases}
\omega(U_\gamma)>0, & 0<\gamma<R_0,\\[1mm]
\omega(U_\gamma)=0, & \gamma=R_0,\\[1mm]
\omega(U_\gamma)<0, & \gamma>R_0.
\end{cases}
\]
\end{lem}

\begin{proof}
For each $\gamma>0$, let $\mathscr L_\gamma$ denote the next
generation operator associated with \eqref{3.28}.
Since $F(t)$ is replaced by $\frac{1}{\gamma}F(t)$, while the evolution
family $\{\Gamma(t,s)\}_{t\geq s}$ associated with
\[
\frac{dv(t)}{dt}=Av(t)-V(t)v(t)
\]
remains unchanged, it follows from the definition of the next
generation operator that
$$
(\mathscr L_\gamma\varphi)(t)
=
\frac{1}{\gamma}F(t)
\lim_{\mu\to+\infty}
\int_{-\infty}^{t}
\Gamma(t,s)\mu(\mu I-A)^{-1}\phi(s)\,ds
=
\frac{1}{\gamma}(\mathscr L\phi)(t).
$$
Thus, we have $\mathscr L_\gamma=\frac{1}{\gamma}\mathscr L$. Therefore, we can deduce
$$
r(\mathscr L_\gamma)=\frac{1}{\gamma}r(\mathscr L)=\frac{R_0}{\gamma}.
$$
By applying Theorem~\ref{TH3.17} to \eqref{3.28} and using Proposition \ref{PR3.6}, we obtain
$$
\operatorname{sign}\big(\omega(U_\gamma)\big)=\operatorname{sign}\big(r(\mathscr L_\gamma)-1\big)=\operatorname{sign}\left(\frac{R_0}{\gamma}-1\right).
$$
Thus, if $0<\gamma<R_0$, then
$$
\omega(U_\gamma)>0.
$$
If $\gamma=R_0$, then
$$
\omega(U_{R_0})=0.
$$
Finally, if $\gamma>R_0$, then
$$
\omega(U_\gamma)<0.
$$
Therefore, $\gamma=R_0$ is the unique positive solution of $\omega(U_\gamma)=0.$
\end{proof}

\begin{remk}
Lemma \ref{LEM-R0-computation} provides a computational characterization of $R_0$ by reducing its computation to the equation
$$
\omega(U_\gamma)=0.
$$
Numerical methods for evaluating such growth bounds have been developed for several classes of almost-periodic models. For example, Wang et al. established the long-term growth characteristics of almost periodic reaction-diffusion systems and computed $R_0$ through a bisection method \cite{Wang23}. Related results can be found in \cite{Qiang,Wang13}. In fact, the specific numerical computation of $R_0$ depends on the structure of the underlying model and the corresponding additional assumption.
\end{remk}
}

In many applications, $F(t)$ always maps $X$ into $X_1$, where $X_1$ is a closed subspace of $X$. Here, we give a corollary to make our results easier to use in such cases.

\begin{coro}\label{CO3.20}
Let Assumptions \ref{ASS2.3}, \ref{ASS2.4} and \ref{ASS3.10} be satisfied. Assume that there exists a closed subspace $X_1\subset X$ such that $F(t)X_0\subset X_1$ for all $t\in \mathbb R$. Then $R_0=r(\mathscr L_*)$, where $\mathscr L_*$ is the restriction of $\mathscr L$ to $AP(\mathbb R,X_1)$. 
\end{coro}
\begin{proof}
Since $F(t)X_0\subset X_1$ for all $t\in \mathbb R$, $\mathscr L$ maps $AP(\mathbb R,X)$ into $AP(\mathbb R,X_1)$. By Gelfand's formula, we can deduce that $\mathscr L$ has the same spectral radius on the spaces $AP(\mathbb R,X)$ and $AP(\mathbb R,X_1)$. Thus, $R_0=r(\mathscr L_*)$, where $\mathscr L_*$ is the restriction of $\mathscr L$ to $AP(\mathbb R,X_1)$. 
\end{proof}

\section{Applications}
\indent\indent In this section, we apply our results to define the basic reproduction ratio $R_0$ for several age-structured models and functional differential systems.

\subsection{An almost periodic age-structured diffusive population model}
\indent\indent Let us consider the following almost periodic population model with spatial diffusion under the Neumann boundary condition. Let $\Omega\subset \mathbb R^n, n\in\mathbb N_+$ be a bounded, open and connected set (domain) with smooth boundary $\partial \Omega$. The age-structured population model is given by 
\begin{equation}\label{4.1}
\left\{ {\begin{array}{*{20}{l}}
{\left( {\frac{\partial }{{\partial t}} + \frac{\partial }{{\partial a}}} \right)P(t,a,x) = d\Delta P(t,a,x)-\mu(t)P(t,a,x),}\quad t\ge t_0,a\in(0,a_+), x\in \Omega,\\
P(t,0,x)=\int_0^{a_+} \beta(t,a)P(t,a,x)da,\quad t\ge t_0,x\in\Omega,\\
P(t_0,a,x)=p_0(a,x)\in L^1_+((0,a_+),C(\overline\Omega)),\\
\frac{\partial P(t,a,x)}{\partial \nu}=0, \quad t\ge t_0,a\in(0,a_+), x\in \partial \Omega,
\end{array}} \right.
\end{equation}
where $P(t,a,x)$ denotes the density of the population of age $a\ge0$ in position $x\in \overline \Omega$ at time $t\ge t_0$. In addition, $d>0$ denotes the diffusion coefficient, $a_+$ denotes the maximal reproductive age, $\beta(t,a)\in AP(\mathbb R,L^\infty_+(0,a_+))$ and $\mu(t)\in AP_+(\mathbb R)$ denote the fertility and mortality rates respectively. In addition, we assume that there exists $\epsi_0>0$ such that $\mu(t)>\epsi_0$ and $\beta(t,a)\ge \epsi_0$ for all $t\in \mathbb R$ and almost every $a\in(0,a_+)$.
\begin{remk}
If we consider the almost periodic and age-structured population model \eqref{4.1} without spatial diffusion, the threshold dynamics and global asymptotic behavior have been obtained in \cite{Huo25}.  
\end{remk}
Let $E:=C(\overline \Omega)$ and $\mathcal X:=L^1((0,a_+),E)$ endowed with the following norm
$$
\|\phi\|_E=\sup_{x\in \overline\Omega} |\phi(x)|,\quad \|\varphi\|_{\mathcal X}=\int_0^{a_+} \|\varphi(a,x)\|_Eda,\quad \forall \phi\in E, \varphi\in \mathcal X.
$$
Recall that $d\Delta$ is the Laplace operator with the Neumann boundary condition, then
\[D(d\Delta):=\{\phi\in C^2(\Omega)\cap C^1(\overline\Omega):d\Delta \phi\in C( \Omega),\frac{\partial \phi}{\partial \nu}=0 \text{ for }x\in\partial \Omega\}.\]
By \cite[Chapter 7]{Smith95}, we know that the closure of $d\Delta$ generates an analytic semigroup of linear operators $T_{d\Delta}(t)$ on $E$. 

The extended space $X$ and its subspace $X_0$ are defined by
$$
X:=E\times \mathcal X,\quad X_0:=\{0_E\}\times \mathcal X.
$$
For any $(\phi,\varphi)\in X$, the norm is defined by
$$
\|(\phi,\varphi)\|_X:=\|\phi\|_E+\|\varphi\|_{\mathcal X}.
$$
Then we consider the family of bounded linear operators $\{R_\lambda\}_{\lambda>0}$ on $X$, defined by
\begin{equation}\label{4.2}
    {R_\lambda }\left( \begin{array}{l}
\phi \\
\varphi
\end{array} \right) = \left( \begin{array}{l}
0\\
\psi
\end{array} \right) \Leftrightarrow \psi (a) = {e^{ -a\lambda}}T_{d\Delta }(a)\phi  + \int_0^a {{e^{ - \int_s^a {\lambda dl} }}T_{d\Delta }(a - s)\varphi (s)ds.} 
\end{equation}
Observe that $\{R_\lambda\}_{\lambda>0}$ is a pseudo-resolvent on $X$. That is to say that
$$R_\lambda-R_\mu=(\mu-\lambda)R_\lambda R_\mu.$$
Moreover, we have
\[R_\lambda x=0, x\in X \Rightarrow x\in X_0,\]
and
\[\mathop{\lim}\limits_{\lambda\to+\infty}\lambda R_\lambda x=x, \quad \forall x\in X_0.\]
By Section 1.9 of \cite{Pazy}, we can deduce that there exists a unique closed linear operator $A$ that satisfies
$$
A:D(A)\subset X\to X, \overline{D( A)}=X_0, R_\lambda=(\lambda I-A)^{-1},\quad\forall\lambda>0,
$$
Therefore, we can define $A$ as the following form
\begin{equation}\label{4.3}
   A\left( \begin{array}{l}
0\\
\psi
\end{array} \right)=\left( \begin{array}{l}
A_1 \psi \\
A_2 \psi
\end{array} \right).
\end{equation}
\begin{remk}\label{RE4.2}
For a better explanation of $A$, we give the definition form of $A$ here,
\[
    A\left( \begin{array}{l}
0\\
\psi
\end{array} \right)=\left( \begin{array}{l}
A_1 \psi \\
A_2 \psi
\end{array} \right)=\left( \begin{array}{c}
-\psi(0) \\
-\frac{\partial \psi(a,x)}{\partial a}+{d\Delta \psi(a,x)}
\end{array} \right),\]
where $D(A)=\{0\}\times W^{1,1}((0,a_+),C(\overline \Omega))$. Here, the definition of $A$ is meant formally only. For a precise definition, we should use resolvent in \eqref{4.2}.
\end{remk}
In addition, we define $F(t):X_0\to X$ and $V(t):X_0\to X$ as follows
$$
 F(t)  \left( \begin{array}{l}
0\\
\psi
\end{array} \right)=\left( \begin{array}{c}
\int_0^{a_+}\beta(t,a)\psi(a,x)da \\
0
\end{array} \right),\quad  V(t)  \left( \begin{array}{l}
0\\
\psi
\end{array} \right)=\left( \begin{array}{c}
0 \\
\mu(t)\psi(a,x)
\end{array} \right).
$$
By the above notations, the population model \eqref{4.1} can be rewritten in the following form
$$
\frac{du(t)}{dt}=Au(t)+F(t)u(t)-V(t)u(t), \quad u(t_0)=u_0\in X_0
$$
By \cite[Section 6]{Thieme09} or \cite[Section 3.2]{Huo23}, we know that $A$ is a resolvent-positive operator. Recall that almost periodic functions are bounded, then there exists $\lambda>0$ such that $\lambda I-V$ is positive. Moreover, it is easy to see that $F(t)\phi\in X_+$ for all $t\in \mathbb R$ and $\phi\in X_{0}^+$. Therefore, Assumption \ref{ASS2.3} is satisfied in model \eqref{4.1}.

Next, we examine Assumptions \ref{ASS2.4} and \ref{ASS3.10} for the population model \eqref{4.1}.
\begin{lem}\label{LE4.3}
Let $\{\Gamma(t,s)\}_{t\ge s}$ be the evolution family of the non-densely defined Cauchy problem
\begin{equation}\label{4.4}
    \frac{dv(t)}{dt}=Av(t)-V(t)v(t),\quad t\in\mathbb R,
\end{equation}
then $\omega(\Gamma)<0$, where $\omega(\Gamma)$ is the exponential growth bound of $\{\Gamma(t,s)\}_{t\ge s}$. Moreover, for any $\epsi>0$ and $\nu>0$, there exists $\eta>0$ and $\delta<0$ such that
\begin{equation}\label{4.5}
    \|\Gamma(t+\tau,s+\tau)-\Gamma(t,s)\|\le \epsi e^{\delta (t-s)}
\end{equation}
for all $t\ge s$ with $t-s\ge \nu$ and all $\tau\in \mathcal T(V,\eta)$.
\end{lem}
\begin{proof}
By setting $v(t)=(0,z(t))^T$, the Cauchy problem \eqref{4.4} is given explicitly by
\begin{equation}\label{4.6}
    \left\{ {\begin{array}{*{20}{l}}
{\left( {\frac{\partial }{{\partial t}} + \frac{\partial }{{\partial a}}} \right)z(t,a,x) = d\Delta z(t,a,x)-\mu(t)z(t,a,x),}\quad t\ge t_0,a\in(0,a_+),x\in \Omega,\\
z(t,0,x)=0,\quad t\ge t_0,x\in \Omega,\\
z(t_0,a,x)=z_0(a,x)\in L^1_+((0,a_+),C(\overline\Omega)),\\
\frac{\partial z(t,a,x)}{\partial \nu}=0, \quad t\ge t_0,a\in(0,a_+),x\in \partial\Omega.
\end{array}} \right.
\end{equation}
By solving \eqref{4.6} along the characteristics, we have
$$
z(t,a,x)= \left\{ {\begin{array}{*{20}{ll}}
e^{-\int_{a-t+t_0}^a \mu(t+s-a)ds}T_{d\Delta}(t-t_0)z_0(a-t+t_0,x),&0\le t-t_0\le a,\\
0,& t-t_0>a.
\end{array}} \right.
$$
This implies that the evolution family $\{\Gamma(t,s)\}_{t\ge s}$ has the following form
$$ \Gamma(t,t_0)  \left( \begin{array}{l}
0\\
z
\end{array} \right)=\left( \begin{array}{c}
0 \\
\Gamma_2(t,t_0)z
\end{array} \right),\quad \forall t\ge t_0,
$$
where $\{\Gamma_2(t,s)\}_{t\ge s}$ admits the following expression
$$
(\Gamma_2(t,t_0) z)(a,x)= \left\{ {\begin{array}{*{20}{ll}}
e^{-\int_{a-t+t_0}^a \mu(t+s-a)ds}T_{d\Delta}(t-t_0)z_0(a-t+t_0,x),&0\le t-t_0\le a,\\
0,& t-t_0>a.
\end{array}} \right.
$$
Recall that $\{T_{d\Delta}(t)\}_{t\ge 0}$ is the $C_0-$semigroup generated by the Laplace operator with Neumann boundary condition. It is well known that $\omega(T_{d\Delta})\le 0$. Thus, we can deduce that $\omega(\Gamma_2)<0$ and further obtain $\omega(\Gamma)<0$. 

In the following, we consider the almost periodic reaction-diffusion equation under the Neumann boundary condition

\begin{equation}\label{4.7}
    \left\{ {\begin{array}{*{20}{l}}
\frac{\partial \varphi(t,x)}{\partial t}=d\Delta \varphi(t,x)-\mu(t)\varphi(t,x),\quad t\ge t_0,x\in \Omega,\\
\varphi(t_0,x)=\varphi(x), \quad x\in \Omega,\\
\frac{\partial \varphi(t,x)}{\partial \nu}=0,\quad t\ge t_0,x\in\partial \Omega.
\end{array}} \right.
\end{equation}
Based on the theory of almost periodic parabolic equations\cite{Yagi}, \eqref{4.7} admits an evolution family $\{\Gamma_*(t,s)\}_{t\ge s}$ on $C(\overline\Omega)$. Moreover, $\{\Gamma_*(t,s)\}_{t\ge s}$ satisfies
$$
\Gamma_*(t,s)=e^{-\int_s^t\mu(r)dr}T_{d\Delta}(t-s),\quad \forall t\ge s.
$$
For equation \eqref{4.7}, the associated operator family is given by $D(t)=d\Delta -\mu(t)$. Since the  $D(t)$ and $V(t)$ is determined by the same coefficient $\mu(t)$, the corresponding almost periods can be chosen from $\mathcal T(V,\eta)$. According to \cite[Lemma 3.1]{Wang23}, we know that for any $\epsi>0$ and $\nu>0$, there exists $\eta>0$ and $\delta<0$ such that
\begin{equation}
    \|\Gamma_*(t+\tau,s+\tau)-\Gamma_*(t,s)\|\le \epsi e^{\delta (t-s)}
\end{equation}
for all $t\ge s$ with $t-s\ge \nu$ and all $\tau\in \mathcal T(V,\eta)$.
By the expression of $\{\Gamma_2(t,s)\}_{t\ge s}$, it is easy to see that $\{\Gamma_2(t,s)\}_{t\ge s}$ also satisfies the above estimate. Hence, by the definition of $\{\Gamma(t,s)\}_{t\ge s}$, it is clear that \eqref{4.5} is true.
\end{proof}

Therefore, we can apply our results to the age-structured diffusive population model \eqref{4.1}. By Theorem \ref{TH3.17}, the next generation operator can be defined by
$$
\mathscr L(\phi)(t):=F(t)\lim_{\mu\to +\infty} \int_{-\infty}^t \Gamma(t,s)\mu(\mu I-A)^{-1}\phi(s)ds,\quad \forall t\in \mathbb R,\phi\in AP(\mathbb R,X).
$$
Let $X_1:=E\times \{0_{\mathcal X}\}$. Then we find that $F(t)$ maps $X_0$ into $X_1$. It follows from Corollary \ref{CO3.20} that we can define the next generation operator by the restriction of $\mathscr L$ to $AP(\mathbb R,X_1)$. In addition, by Lemma \ref{LE3.9}, we know that $u(t)=\lim_{\lambda\to+\infty}\int_{-\infty}^t \Gamma(t,s) \lambda(\lambda I-A)^{-1}\phi(s)ds$
is the solution of the following equation
\begin{equation}\label{4.8}
    \left\{ {\begin{array}{*{20}{l}}
\frac{du(t)}{dt}=Au(t)-V(t)u(t)+\phi(t), \quad t\ge t_0\\
u(t_0)=u_0.
\end{array}} \right.
\end{equation}
By letting $\phi=(m,0)\in AP(\mathbb R,X_1)$ and $u(t)=(0,z_*(t))^T$, then \eqref{4.8} is given explicitly by
\begin{equation}\label{4.9}
    \left\{ {\begin{array}{*{20}{l}}
{\left( {\frac{\partial }{{\partial t}} + \frac{\partial }{{\partial a}}} \right)z_*(t,a,x) = d\Delta z_*(t,a,x)-\mu(t)z_*(t,a,x),}\quad t\ge t_0,a\in(0,a_+),x\in \Omega,\\
z_*(t,0,x)=m(t,x),\quad t\ge t_0,x\in \Omega,\\
z_*(t_0,a,x)=z_{*0}(a,x)\in L^1_+((0,a_+),C(\Omega)),\\
\frac{\partial z_*(t,a,x)}{\partial \nu}=0,\quad t\ge t_0,a\in(0,a_+),x\in \partial \Omega.
\end{array}} \right.
\end{equation}
By solving along the characteristics again, we have
$$
z_*(t,a,x)= \left\{ {\begin{array}{*{20}{ll}}
e^{-\int_{a-t+t_0}^a \mu(s-a+t_0)ds}T_{d\Delta}(t-t_0)z_{*0}(a-t+t_0,x),&0\le t-t_0\le a,\\
e^{-\int_0^a\mu(t+s-a)ds}T_{d\Delta}(a)m(t-a),& t-t_0>a.
\end{array}} \right.
$$
By setting $t_0\to-\infty$, the entire solution of \eqref{4.9} is given by
$$
z_*(t,a,x)=e^{-\int_0^a\mu(t+s-a)ds}T_{d\Delta}(a) m(t-a),\quad t\in \mathbb R,a\ge 0,x\in \Omega.
$$
Based on Lemma \ref{LE3.9} and $F(t):X_0\to X_1$, then for any $t\in \mathbb R$ and $\phi=(m,0)^T\in AP(\mathbb R,X_1)$, we have
$$
\begin{array}{*{20}{rl}}
\mathscr L(\phi)(t)&=F(t)\lim_{\mu\to +\infty} \int_{-\infty}^t \Gamma(t,s)\mu(\mu I-A)^{-1}\phi(s)ds \\
&=\left( \begin{array}{c}
\int_0^{a_+} \beta(t,a)e^{-\int_0^a \mu(t+s-a)ds}T_{d\Delta}(a)m(t-a)da\\
0
\end{array} \right).
\end{array}
$$
According to Corollary \ref{CO3.20}, we can define the next generation operator by the following form
\begin{equation}\label{4.10}
    \mathscr L_*\psi(t):=\int_0^{a_+} \beta(t,a)e^{-\int_0^a \mu(t+s-a)ds}T_{d\Delta}(a)\psi(t-a)da,\quad \psi\in AP(\mathbb R,C(\overline\Omega))
\end{equation}
As a consequence, the threshold dynamics of the almost periodic age-structured diffusive population model \eqref{4.1} can be determined by 
$$
R_0=r(\mathscr L_*).
$$
By using Corollary \ref{CO3.18}, we can deduce the following results.
\begin{thm}\label{TH4.4}
Let $\{U(t,s)\}_{t\ge s}$ be the evolution family of the population model \eqref{4.1} and $\omega(\cdot)$ be the exponential growth bound, then the following statements are valid.
\begin{itemize}
\item[{\rm (i)}] If $r(\mathscr L_*)<1$, then $\omega(U)<0$ and $\lim_{t\to +\infty}\|P(t,\cdot,\cdot)\|=0$.
\item[{\rm (ii)}]  If $r(\mathscr L_*)=1$, then $\omega(U)=0$.
\item[{\rm (iii)}]  If $r(\mathscr L_*)>1$, then $\omega(U)>0$.
\end{itemize}  
\end{thm}
\begin{remk}\label{RE4.5}
Since almost periodic functions are the generalization of periodic functions, the threshold dynamics of the almost periodic population model \eqref{4.1} can be directly applied to the periodic case. More precisely, let $C_p$ denote the space of the periodic functions, $\beta(t,a)\in C_p(\mathbb R,L^\infty_+(0,a_+))$ and $\mu(t)\in C_p(\mathbb R)$, then the next generation operator $\mathscr L_*$ of the periodic population model \eqref{4.1} can also be defined by \eqref{4.10}. The difference is that $\mathscr L_*$ should be defined on $C_p(\mathbb R, C(\overline\Omega))$, instead of $AP(\mathbb R,C(\overline\Omega))$. Moreover, Theorem \ref{TH4.4} also holds in the periodic case.
\end{remk}

{
\subsection{An almost periodic and age-structured population model with nonlocal diffusion}
\indent\indent In this subsection, we extend the definition of  $R_0$ from the local case to the nonlocal diffusion case for an age-structured population model. Let $\Omega\subset\mathbb R^n$ be a bounded, open and connected domain and $n\ge1$ be an integer. We consider the following model
\begin{equation}\label{4.01}
\left\{ {\begin{array}{*{20}{l}}
{\left( {\frac{\partial }{{\partial t}} + \frac{\partial }{{\partial a}}} \right)P(t,a,x) =d(J*P-P)(t,a,x) - \mu (t)P(t,a,x),\quad t\ge t_0, a\in(0,a_+), x\in \Omega,}\\
{P(t,0,x) = \int_0^{a_+}  {\beta (t,a)P(t ,a,x)da}  ,\quad t\ge t_0, x\in \Omega,}\\
{P(t_0,a,x)=p_0(a,x)\in L^1_+((0,a_+),L^2(\Omega))}\\
{P(t,a,x) = 0,\quad t\ge t_0, a\in(0,a_+), x \notin\Omega.}
\end{array}} \right.
\end{equation}
where $P(t,a,x)$ denotes the density of the population of age $a\ge0$ in position $x\in \overline \Omega$ at time $t\ge t_0$. The diffusion kernel $J$ is a $C_0$, compactly supported, non-negative function with unit integral representing the spatial dispersal, i.e.
$$\int_{{\mathbb R^n}}J(x)dx=1, J(x)\ge0,\quad \forall x\in \mathbb R^n,$$
and $J*P-P$ is defined by
$$(J*P-P)(t,a,x)=\int_{\mathbb R^n}J(x-y)P(t,a,y)dy-P(t,a,x).$$
In addition, $d$ denotes the diffusion coefficient, $a_+$ denotes the least upper bound of reproductive age, $\beta(t,a)\in AP(\mathbb R,L^\infty_+(0,a_+))$ and $\mu(t)\in AP_+(\mathbb R)$ denote the fertility and mortality rates respectively. In addition, we assume that there exists $\epsi_0>0$ such that $\mu(t)>\epsi_0$ and $\beta(t,a)\ge \epsi_0$ for all $t\in \mathbb R$ and $a\in(0,a_+)$.

Let $E:=L^2(\Omega)$ and $\mathcal X:=L^1((0,a_+),E)$ endowed with the following norm
$$
\|\phi\|_E=\left(\int_\Omega \phi^2(x) dx\right)^{\frac{1}{2}},\quad \|\varphi\|_{\mathcal X}=\int_0^{a_+} \|\varphi(a,x)\|_Eda,\quad \forall \phi\in E, \varphi\in \mathcal X.
$$
By \cite{Kang2021}, the nonlocal diffusion operator $d(J*\cdot-\cdot)$ with the homogeneous Dirichlet condition generates a positive $C_0$-semigroup $\{\mathbb T(t)\}_{t\geq0}$ on $E=L^2(\Omega)$. That is, $\{\mathbb T(t)\}_{t\ge 0}$ is the solution semigroup on $E$ of the following non-local diffusion equation
$$
\frac{\partial u(a,x)}{\partial a}=d(J*u-u)(a,x).
$$

The extended space $X$ and its subspace $X_0$ are defined by
$$
X:=E\times \mathcal X,\quad X_0:=\{0_E\}\times \mathcal X.
$$
For any $(\phi,\varphi)\in X$, the norm is defined by
$$
\|(\phi,\varphi)\|_X:=\|\phi\|_E+\|\varphi\|_{\mathcal X}.
$$
Then we consider the family of bounded linear operators $\{R_\lambda\}_{\lambda>0}$ on $X$, defined by
\begin{equation}\label{4.02}
    {R_\lambda }\left( \begin{array}{l}
\phi \\
\varphi
\end{array} \right) = \left( \begin{array}{l}
0\\
\psi
\end{array} \right) \Leftrightarrow \psi (a) = {e^{ -a\lambda}}\mathbb T(a)\phi  + \int_0^a {{e^{ - \int_s^a {\lambda dl} }}\mathbb T(a - s)\varphi (s)ds.} 
\end{equation}
Similar to the argument in subsection 4.1, we can deduce that there exists a unique closed linear operator $A$ that satisfies
$$
A:D(A)\subset X\to X, \overline{D( A)}=X_0, R_\lambda=(\lambda I-A)^{-1},\quad\forall\lambda>0,
$$
In addition, we can define $F(t):X_0\to X$ and $V(t):X_0\to X$
$$
 F(t)  \left( \begin{array}{l}
0\\
\psi
\end{array} \right)=\left( \begin{array}{c}
\int_0^{a_+}\beta(t,a)\psi(a,x)da \\
0
\end{array} \right),\quad  V(t)  \left( \begin{array}{l}
0\\
\psi
\end{array} \right)=\left( \begin{array}{c}
0 \\
\mu(t)\psi(a,x)
\end{array} \right).
$$
By the above notations, the population model \eqref{4.01} can be rewritten in the following form
$$
\frac{du(t)}{dt}=Au(t)+F(t)u(t)-V(t)u(t), \quad u(t_0)=u_0\in X_0.
$$
As in subsection 4.1, we can deduce that Assumption \ref{ASS2.3} holds. Next, we examine Assumptions \ref{ASS2.4} and \ref{ASS3.10}.

\begin{lem}\label{LE4.01}
Let $\{\Gamma(t,s)\}_{t\ge s}$ be the evolution family of the non-densely defined Cauchy problem
\begin{equation}\label{4.03}
    \frac{dv(t)}{dt}=Av(t)-V(t)v(t),\quad t\in\mathbb R,
\end{equation}
then $\omega(\Gamma)<0$, where $\omega(\Gamma)$ is the exponential growth bound of $\{\Gamma(t,s)\}_{t\ge s}$. Moreover, there exist $C>0$ and
$\delta<0$ such that
\begin{equation}\label{4.04}
\|\Gamma(t+l,s+l)-\Gamma(t,s)\|\le C\eta e^{\delta(t-s)}
\end{equation}
for all $t\geq s$, $\eta>0$, and $l\in\mathcal T(V,\eta)$.
\end{lem}
\begin{proof}
Similar to Lemma \ref{LE4.3}, by solving \eqref{4.03} along the characteristics, we have
$$ \Gamma(t,t_0)  \left( \begin{array}{l}
0\\
z
\end{array} \right)=\left( \begin{array}{c}
0 \\
\Gamma_2(t,t_0)z
\end{array} \right),\quad \forall t\ge t_0,
$$
where $\{\Gamma_2(t,s)\}_{t\ge s}$ admits the following expression
$$
(\Gamma_2(t,t_0) z)(a,x)= \left\{ {\begin{array}{*{20}{ll}}
e^{-\int_{a-t+t_0}^a \mu(t+s-a)ds}\mathbb T(t-t_0)z_0(a-t+t_0,x),&0\le t-t_0\le a,\\
0,& t-t_0>a.
\end{array}} \right.
$$
Since $\|\mathbb T(t)\|\leq1$ and
$\mu(t)\geq\epsi_0$, we obtain $\|\Gamma(t,s)\|\le e^{-\epsi_0(t-s)},\forall t\geq s$, and hence $\omega(\Gamma)\leq-\epsi_0<0$.

Let $l\in\mathcal T(V,\eta)$. Then
$$
|\mu(r+l)-\mu(r)|<\eta,\quad r\in\mathbb R.
$$
Using the above expression of $\Gamma(t,s)$, we obtain
$$
\|\Gamma(t+l,s+l)-\Gamma(t,s)\|\le\eta(t-s)e^{-\epsi_0(t-s)}.
$$
Since $(t-s)e^{-\frac{\epsi_0}{2}(t-s)}\le \frac{2}{\epsi_0 e}$ and the maximum is attained at \(t-s=\frac{2}{\epsi_0}\), we have
$$
\|\Gamma(t+l,s+l)-\Gamma(t,s)\|
\le \frac{2\eta}{\epsi_0e}
e^{-\frac{\epsi_0}{2}(t-s)}.
$$
Thus, \eqref{4.04} holds with $C=\frac{2}{\epsi_0e},\delta=-\frac{\epsi_0}{2}<0.$
\end{proof}

Therefore, the abstract results developed in Section 3 can be applied to model \eqref{4.01}. Let $X_1:=E\times\{0_{\mathcal X}\}$. Since $F(t)$ maps $X_0$ into $X_1$, by Corollary \ref{CO3.20}, the next generation operator can be restricted to $AP(\mathbb R,X_1)$. Following the steps in Section 4.1, the next generation operator $\mathscr L_*$ has the following form
$$
   \mathscr L_*\psi(t):=\int_0^{a_+} \beta(t,a)e^{-\int_0^a \mu(t+s-a)ds}\mathbb T(a)\psi(t-a)da,\quad \psi\in AP(\mathbb R,L^2(\Omega))
$$
Moreover, the basic reproduction ratio for model \eqref{4.01} is defined by 
$$
R_0=r(\mathscr L_*).
$$
By Corollary \ref{CO3.18}, we immediately obtain the following result.
\begin{thm}
Let $\{U(t,s)\}_{t\ge s}$ be the evolution family of the population model \eqref{4.01} and $\omega(\cdot)$ be the exponential growth bound, then the following statements are valid.
\begin{itemize}
\item[{\rm (i)}] If $r(\mathscr L_*)<1$, then $\omega(U)<0$ and $\lim_{t\to +\infty}\|P(t,\cdot,\cdot)\|=0$.
\item[{\rm (ii)}]  If $r(\mathscr L_*)=1$, then $\omega(U)=0$.
\item[{\rm (iii)}]  If $r(\mathscr L_*)>1$, then $\omega(U)>0$.
\end{itemize}  
\end{thm}

}

\subsection{An almost periodic and infection age-structured SIR model}
\indent\indent In \cite{Rebelo}, Rebelo, Margheri and Bacaër studied the time-periodic infection age-structured SIR model and defined the basic reproduction ratio $R_0$. In this subsection, we extend the definition of $R_0$ to an infection age-structured SIR model in the almost periodic situation. Thus, we consider the following epidemic model. For $t\ge t_0$ and $a\in (0,+\infty)$,
\begin{equation}\label{4.11}
\left\{ {\begin{array}{*{20}{l}}
\frac{dS(t)}{dt}=b(t)-\mu(t)S(t)-S(t)\int_0^\infty \beta(t,a) I(t,a)da, \\
{\left( {\frac{\partial }{{\partial t}} + \frac{\partial }{{\partial a}}} \right)I(t,a) = -\mu(t)I(t,a)-\gamma(t,a)I(t,a),}\\
I(t,0)=S(t)\int_0^{+\infty} \beta(t,a)I(t,a)da,\\
\frac{dR(t)}{dt}=\int_0^{+\infty}\gamma(t,a)I(t,a)da-\mu(t)R(t),
\end{array}} \right.
\end{equation}
with the initial value condition
\begin{equation}
S(t_0)=S_0,I(t_0,a)=I_0(a),R(t_0)=R_0,\quad a\in(0,\infty).
\end{equation}
In the epidemic model \eqref{4.11}, $S(t)$ and $R(t)$ denote the densities of susceptible and removed individuals at time $t$, respectively. $I(t,a)$ denotes the density of infective individuals with infection age $a$ at time $t$. In addition, $b(t)\in AP_+(\mathbb R,\mathbb R)$ is the fertility rate,  $\beta(t,a)\in AP(\mathbb R,L^\infty_+(0,+\infty))$ is the  transmission rate, $\mu(t)\in AP_+(\mathbb R,\mathbb R)$ is the mortality rate and $\gamma(t,a)\in AP(\mathbb R,L^\infty_+(0,+\infty))$ is the removal rate. We assume that there exists $\epsi_0>0$ such that $b(t)\ge \epsi_0, \mu(t)\ge \epsi_0$ and $\gamma(t,a)> \epsi_0$ for all $t\in \mathbb R$ and $a\in(0,+\infty)$.

{ By the standard theory of almost periodic differential equations \cite{Fink}, the equation
$$
\frac{dS(t)}{dt}=b(t)-\mu(t)S(t)
$$
admits a unique positive almost periodic solution
$$S_0(t)=\int_{-\infty}^tb(s)e^{-\int_s^t\mu(\xi)d\xi}ds.$$
Since $\mu(t)\ge \epsi_0>0$, for any solution $S(t)$, we have
$$
|S(t)-S_0(t)|\le|S(t_0)-S_0(t_0)|e^{-\epsi_0(t-t_0)},\quad t\ge t_0.
$$
Therefore, $S_0(t)$ is globally exponentially attractive.} This implies that the epidemic model \eqref{4.11} admits a unique almost periodic disease-free solution $(S_0(t),0,0)$.  In the following, we consider the linearization of the $I-$equation at almost periodic disease-free solution $(S_0(t),0,0)$, that is
\begin{equation}\label{4.13}
\left\{ {\begin{array}{*{20}{l}}
{\left( {\frac{\partial }{{\partial t}} + \frac{\partial }{{\partial a}}} \right)I(t,a) = -\mu(t)I(t,a)-\gamma(t,a)I(t,a),}\\
I(t,0)=S_0(t)\int_0^{+\infty} \beta(t,a)I(t,a)da.
\end{array}} \right.
\end{equation}
Let $X:=\mathbb R\times L^1(0,+\infty)$ and its closed subspace $X_0=\{0_{\mathbb R}\}\times L^1(0,+\infty)$. We define $A:D(A)\subset X\to X$ by 

\begin{equation}\label{4.14}
   A\left( \begin{array}{l}
0\\
\psi
\end{array} \right)=\left( \begin{array}{c}
-\psi(0) \\
-\frac{d\psi}{da}-\epsi_0\psi
\end{array} \right),\quad D(A):=\{0_{\mathbb R}\}\times W^{1,1}(0,+\infty).
\end{equation}
In addition, we define $F(t):X_0\to X$ and $V(t):X_0\to X$ as follows

$$
 F(t)  \left( \begin{array}{l}
0\\
\psi
\end{array} \right)=\left( \begin{array}{c}S_0(t)\int_0^{+\infty}\beta(t,a)\psi(a)da \\
0
\end{array} \right),\quad  V(t)  \left( \begin{array}{l}
0\\
\psi
\end{array} \right)=\left( \begin{array}{c}
0 \\
\mu(t)\psi(a)+\gamma(t,a)\psi(a)-\epsi_0\psi(a)
\end{array} \right).
$$
Thus, the epidemic model \eqref{4.11} can be rewritten as the following non-densely defined Cauchy problem
\begin{equation}\label{4.15}
    \frac{du(t)}{dt}=Au(t)+F(t)u(t)-V(t)u(t), \quad u(0)=u_0\in X_0
\end{equation}
It is easy to check that Assumption \ref{ASS2.3} holds on the Cauchy problem \eqref{4.15}. Moreover, following the ideas in Lemma \ref{LE4.3}, we can deduce that Assumptions \ref{ASS2.4} and \ref{ASS3.10} are satisfied. 

Let $X_1:=\mathbb R\times \{0_{L^1(0,+\infty)}\}$. We can see that $F(t)$ maps $X_0$ into $X_1$. Similar to the argument in subsection  4.1, we consider the following Cauchy problem
\begin{equation}\label{4.16}
\frac{du(t)}{dt}=Au(t)-V(t)u(t)+\phi(t), \quad u(t_0)=u_0.
\end{equation}
By letting $\phi=(m,0)\in AP(\mathbb R,X_1)$ and $u(t)=(0_{\mathbb R},z_*(t))^T$, then \eqref{4.16} is given explicitly by
\begin{equation}\label{4.17}
    \left\{ {\begin{array}{*{20}{l}}
{\left( {\frac{\partial }{{\partial t}} + \frac{\partial }{{\partial a}}} \right)z_*(t,a) =-\mu(t)z_*(t,a)-\gamma(t,a)z_*(t,a),}\quad t\ge t_0,a\in(0,+\infty)\\
z_*(t,0)=m(t),\quad t\ge t_0\\
z_*(t_0,a)=z_{*0}(a)\in L^1_+(0,+\infty),
\end{array}} \right.
\end{equation}
By solving along the characteristics again, we have
$$
z_*(t,a)= \left\{ {\begin{array}{*{20}{ll}}
e^{-\int_{a-t+t_0}^a \mu(s-a+t_0)+\gamma(s-a+t_0,s)ds}   z_{*0}(a-t+t_0),&0\le t-t_0\le a,\\
e^{-\int_0^a\mu(t+s-a)+\gamma(t+s-a,s) ds}  m(t-a),& t-t_0>a.
\end{array}} \right.
$$
Letting $t_0\to -\infty$, we have
$$
z_*(t,a)=e^{-\int_0^a\mu(t+s-a)+\gamma(t+s-a,s) ds}  m(t-a).
$$
By using Lemma \ref{LE3.9}, we obtain

$$
\begin{array}{*{20}{rl}}
\mathscr L(\phi)(t)&=F(t)\lim_{\mu\to +\infty} \int_{-\infty}^t \Gamma(t,s)\mu(\mu I-A)^{-1}\phi(s)ds \\
&=\left( \begin{array}{c}
S_0(t)\int_0^{+\infty} \beta(t,a)e^{-\int_0^a\mu(t+s-a)+\gamma(t+s-a,s) ds}  m(t-a)da\\
0
\end{array} \right).
\end{array}
$$
Recall that $F(t)$ maps $X_0$ into $X_1$, it follows from Corollary \ref{CO3.20} that the next generation operator of epidemic model \eqref{4.11} can be defined by 
$$
(\mathscr L_* \psi)(t)=S_0(t)\int_0^{+\infty} \beta(t,a)e^{-\int_0^a\mu(t+s-a)+\gamma(t+s-a,s) ds}  \psi(t-a)da,\quad \forall \psi\in AP(\mathbb R,\mathbb R).
$$
Moreover, we can define the basic reproduction ratio $R_0$ by 
$$R_0=r(\mathscr L_*).$$
\begin{remk}\label{RE4.6}
In \cite{Rebelo}, Rebelo, Margheri and Bacaër studied the time-periodic infection age-structured SIR model and defined the basic reproduction ratio $R_0$. Later, Diagne, Seydi and Sy extended the results to a two-group case \cite{Diagne}. It is easy to see that the next generation operator has the same expression in the periodic and almost periodic cases, but is defined on different Banach spaces. It is worth noting that the age-structured and diffusive equation can be transferred into a non-densely defined Cauchy problem. Therefore, by a similar argument in subsection 4.1, we can define the basic reproduction ratio $R_0$ for the almost periodic and infection age-structured epidemic models with spatial diffusion.
\end{remk}

\subsection{Almost periodic compartmental models with time delay}
\indent\indent In this subsection, we consider the following linear functional differential system
\begin{equation}\label{4.18}
\left\{ {\begin{array}{*{20}{l}}
\frac{dx(t)}{dt}=\mathcal F(t)x_t -\mathcal V(t)x(t), \quad t\ge t_0\\
x_{t_0}=\varphi\in C([-\tau,0],\mathbb R^n),
\end{array}} \right.    
\end{equation}
where $\tau>0$ is a constant denoting the maximal delay, $\mathcal F(t): C([-\tau,0],\mathbb R^n)\to \mathbb R^n$ and $\mathcal V(t):\mathbb R^n\to \mathbb R^n$ are almost periodic with respect to $t$ in the operator norm topology. Let $C:=C([-\tau,0],\mathbb R^n)$ be the Banach space of continuous functions from $[-\tau, 0]$ to $\mathbb R^n$ endowed with the supremum norm
$$
\|\varphi\|_C=\sup_{\theta\in[-\tau,0]}\|\varphi(\theta)\|_{\mathbb R^n},\quad \forall\varphi\in C.
$$

System \eqref{4.18} can be regarded as the equations of infectious variables in the linearization of an almost periodic and time-delayed compartmental epidemic model at an almost periodic disease-free solution. $\mathcal F(t)x_t$ denotes the newly infected individuals at time $t$ depend linearly on the infectious individuals over the time interval $[t-\tau, t]$. $\mathcal V(t)$ is the sum of the rate of birth, out and transfer of infected individuals. The basic reproduction ratio and the threshold dynamics of $\eqref{4.18}$ have been given in \cite{Qiang20}. Here, we use the non-densely defined operators to study this problem again and obtain the same results as in \cite{Qiang20}.

For system \eqref{4.18}, we make the following assumptions,

 \noindent \textbf{(A1)} The operator $\mathcal F(t): C([-\tau,0],\mathbb R^n)\to \mathbb R^n$ is positive, that is, $F(t)$ map $C([-\tau,0],\mathbb R^n_+)$ to $\mathbb R^n_+$.

 \noindent \textbf{(A2)} The matrix $-\mathcal V(t)$ is cooperative and $\omega(\Pi)<0$, where $\omega(\Pi)$ is the exponential growth bound of $\{\Pi(t,s)\}_{t\ge s}$ and $\{\Pi(t,s)\}_{t\ge s}$ is the evolution family of the following almost periodic differential equations
 \begin{equation}\label{4.19}
     \frac{dx(t)}{dt}=-\mathcal V(t)x(t).
 \end{equation}

Firstly, we transform the functional differential equation \eqref{4.18} into a non-densely defined Cauchy problem. Let $v\in C([t_0,\infty)\times [-\tau,0],\mathbb R^n)$
$$
v(t,\theta):=x(t+\theta),\quad \forall t\ge t_0,\theta\in[-\tau,0].
$$
Note that if $x\in C^1([t_0-\tau,+\infty),\mathbb R^n)$, we have
$$
\frac{\partial v(t,\theta)}{\partial t}=x'(t+\theta)=\frac{\partial v(t,\theta)}{\partial \theta}.
$$
Thus, we have
$$
\frac{\partial v(t,\theta)}{\partial t}-\frac{\partial v(t,\theta)}{\partial \theta}=0,\quad \forall t\ge t_0,\theta\in[-\tau,0].
$$
In addition, if $\theta=0$, we have
$$
\frac{\partial v(t,0)}{\partial \theta}=x'(t)=\mathcal F(t)x_t-\mathcal V(t)x(t)=\mathcal F(t)v(t,\cdot)-\mathcal V(t)v(t,0).
$$
Therefore, $v$ satisfies the following differential equations
\begin{equation}\label{4.20}
\left\{ {\begin{array}{*{20}{l}}
\frac{\partial}{\partial t}v(t,\theta)-\frac{\partial}{\partial \theta}v(t,\theta)=0,\\
\frac{\partial }{\partial \theta}v(t,0)=\mathcal F(t)v(t,\cdot)-\mathcal V(t)v(t,0),\\
v(t_0,\cdot)=v_0\in C.
\end{array}} \right.
\end{equation}
We introduce an extended space $X$ by 
$$
X:=\mathbb R^n\times C([-\tau,0],\mathbb R^n). 
$$
Define a linear operator $A:D(A)\subset X\to X$ by
\begin{equation}\label{4.21}
   A\left( \begin{array}{l}
0\\
\psi
\end{array} \right)=\left( \begin{array}{c}
-\psi'(0) \\
\psi'
\end{array} \right),\quad D(A):=\{0_{\mathbb R^n}\}\times C^1([-\tau,0],\mathbb R^n).
\end{equation}
Note that $A$ is non-densely defined because
$$
X_0:=\overline{D(A)}=\{0_{\mathbb R^n}\}\times C\neq X.
$$
In addition, we define $F(t):X_0\to X$ and $V(t):X_0\to X$ by
$$
 F(t)  \left( \begin{array}{l}
0\\
\psi
\end{array} \right)=\left( \begin{array}{c} \mathcal F (t)\psi \\
0
\end{array} \right),\quad  V(t)  \left( \begin{array}{l}
0\\
\psi
\end{array} \right)=\left( \begin{array}{c}
\mathcal V(t)\psi(0) \\
0
\end{array} \right).
$$
By using the above notations, \eqref{4.20} has the following form
\begin{equation}\label{4.22}
        \frac{du(t)}{dt}=Au(t)+F(t)u(t)-V(t)u(t), \quad u(t_0)=u_0\in X_0.
\end{equation}

\begin{lem}\label{LE4.7}\cite{Liu08}
Let Assumptions (A1) and (A2) be satisfied. Then $A$ is a Hille-Yosida operator. Moreover, $\rho(A)=\mathbb C\setminus \{0\}$ and  for any $\lambda>0$, we have the following explicit formula for the resolvent of $A$:

\begin{equation}\label{4.23}
     (\lambda I-A)^{-1}\left( \begin{array}{l}
f\\
\varphi
\end{array} \right)=\left( \begin{array}{l}
0\\
\psi
\end{array} \right)  
      \Leftrightarrow \psi(\theta)=e^{\lambda\theta} \frac{\varphi(0)+f}{\lambda}+\int_\theta^0 e^{\lambda(\theta-s)}\varphi(s)ds,\quad \theta\in[-\tau,0].
\end{equation}
Furthermore, the part $A_0$ of $A$ in $X_0$ generates a $C_0-$semigroup $\{T_{A_0}(t)\}_{t\ge 0}$ with the following form
$$
T_{A_0}(t)\left( \begin{array}{l}
0\\
\psi
\end{array} \right) =\left( \begin{array}{c}
0\\
\overline T_{A_0}(t)\psi
\end{array} \right) ,
$$
where $\overline T_{A_0}(t)\psi$ is given by 
$$
(\overline T_{A_0}(t)\psi)(\theta)=\left\{ {\begin{array}{*{20}{ll}}
\psi(t+\theta)&-\tau\le t+\theta\le 0,\\
\psi(0),& t+\theta>0.
\end{array}} \right.
$$
\end{lem}
By \eqref{4.23}, we can see that $A$ is a resolvent positive operator. In addition, based on Assumptions (A1) and (A2), Assumption \ref{ASS2.3} is satisfied. In the following, we show Assumptions \ref{ASS2.4} and \ref{ASS3.10}.

\begin{lem}\label{LE4.8}
Let Assumptions (A1) and (A2) be satisfied. Let $\{\Gamma(t,s)\}_{t\ge s}$ be the evolution family of the following non-densely defined Cauchy problem
\begin{equation}\label{4.24}
    \frac{dv(t)}{dt}=Av(t)- V(t)v(t),\quad t\in \mathbb R.
\end{equation}
Then $\omega(\Gamma)<0$. Moreover,  there exist $C>0$ and $\delta<0$ such that
\begin{equation}\label{4.25}
\|\Gamma(t+l,s+l)-\Gamma(t,s)\|
\leq C\eta e^{\delta(t-s)}
\end{equation}
for all $t\geq s$, $\eta>0$, and $l\in\mathcal T(V,\eta)$.
\end{lem}
\begin{proof}
By \cite[Theorem 3.2]{Rhandi97}, $\{\Gamma(t,s)\}_{t\ge s}$ admits the following variation of constants formula
$$
v(t)=\Gamma(t,t_0)v_0=T_{A_0}(t-t_0)v_0-\lim_{\lambda\to +\infty}\int_{t_0}^t T_{A_0}(t-s)\lambda(\lambda I-A)^{-1} V(s)\Gamma(s,t_0)v_0ds.
$$
By letting $v_0=(0,\psi)$, we have
\begin{equation}\label{4.26}
    \Gamma(t,t_0)\left( \begin{array}{l}
0\\
\psi
\end{array} \right) =\left( \begin{array}{c}
0\\
\overline \Gamma(t,t_0)\psi
\end{array} \right) ,
\end{equation}
with
\begin{equation}\label{4.27}
(\overline\Gamma(t,t_0)\psi)(\theta)=\left\{ {\begin{array}{*{20}{ll}}
\psi(t-t_0+\theta),&-\tau\le t-t_0+\theta\le 0,\\
\Pi(t+\theta,t_0)\psi(0),& t-t_0+\theta>0,
\end{array}} \right.
\end{equation}
where $\{\Pi(t,s)\}_{t\ge s}$ is the evolution family of the almost periodic differential equation
\begin{equation}\label{4.28}
    \frac{dw(t)}{dt}=- \mathcal V(t)w(t),\quad t\ge t_0.
\end{equation}
Due to Assumption (A2), we know $\omega(\Pi)<0$. Hence, there exist $M\geq1$ and $\alpha>0$ such that $\|\Pi(t,s)\|\leq Me^{-\alpha(t-s)},\quad t\geq s$. It follows from \eqref{4.27} that $\omega(\Gamma)<0$. Let $l\in\mathcal T(V,\eta)$. Then
$$
\|\mathcal V(r+l)-\mathcal V(r)\|<\eta,\quad r\in\mathbb R.
$$
By the variation-of-constants formula for \eqref{4.28},
$$
\Pi(t+l,s+l)-\Pi(t,s)=-\int_s^t\Pi(t+l,r+l)[\mathcal V(r+l)- \mathcal V(r)]\Pi(r,s)dr.
$$
Thus, we have
$$
\|\Pi(t+l,s+l)-\Pi(t,s)\|\leq M^2\eta(t-s)e^{-\alpha(t-s)}\leq
\frac{2M^2\eta}{\alpha e}e^{-\frac{\alpha}{2}(t-s)}.
$$
By \eqref{4.27}, for $\theta\in[-\tau,0]$, the difference is zero
when $t-s+\theta\leq0$, while for $t-s+\theta>0$,
$$
\|(\overline\Gamma(t+l,s+l)-\overline\Gamma(t,s))\psi(\theta)\|\leq\frac{2M^2\eta}{\alpha e}e^{-\frac{\alpha}{2}(t-s+\theta)}\|\psi\|\leq\frac{2M^2\eta}{\alpha e}e^{\frac{\alpha\tau}{2}}e^{-\frac{\alpha}{2}(t-s)}\|\psi\|.
$$
Therefore, we have
$$
\|\Gamma(t+l,s+l)-\Gamma(t,s)\|
\leq
C\eta e^{\delta(t-s)},
$$
where $C=\frac{2M^2}{\alpha e}e^{\frac{\alpha\tau}{2}},\delta=-\frac{\alpha}{2}<0.$
\end{proof}

Based on the definition of $F(t)$, we can find that $F(t)$ maps $X_0$ into $X_1$, where $X_1:={\mathbb R^n}\times \{0_C\}$. In order to use Corollary \ref{CO3.20}, we calculate the restriction of $\mathscr L$ to $AP(\mathbb R,X_1)$,
$$
   \mathscr L\left( \begin{array}{l}
f\\
0
\end{array} \right)(t)=F(t)\lim_{\lambda\to +\infty}\int_{-\infty}^t\Gamma(t,s)\lambda(\lambda I-A)^{-1}\left( \begin{array}{c}
f(s) \\
0
\end{array} \right)ds, \quad t\in\mathbb R, f\in AP(\mathbb R,\mathbb R^n).
$$
By \eqref{4.23}, we have
$$
\lambda(\lambda I-A)^{-1}\left( \begin{array}{l}
f\\
0
\end{array} \right)=\left( \begin{array}{l}
0\\
\psi
\end{array} \right) ,\quad \forall\lambda>0
$$
where
$$
\psi(\theta)=e^{\lambda\theta} f(s),\quad \theta\in[-\tau,0],\lambda>0.
$$
From \eqref{4.26}, we can obtain
$$
\int_{-\infty}^t \Gamma(t,s)\lambda(\lambda I-A)^{-1} \left( \begin{array}{c}
f(s)\\
0
\end{array} \right)ds=\left( \begin{array}{c}
0\\
\int_{-\infty}^t \overline\Gamma(t,s) e^{\lambda\theta} f(s)ds
\end{array} \right),\quad t\in \mathbb R.
$$
Therefore, the restriction of $\mathscr L$ to $AP(\mathbb R,X_1)$ admits the following form
$$
   \mathscr L\left( \begin{array}{l}
f\\
0
\end{array} \right)(t)=\left( \begin{array}{c}
\mathcal F(t) \lim_{\lambda\to +\infty}\int_{-\infty}^t \overline\Gamma(t,s) e^{\lambda\theta} f(s)\\
0
\end{array} \right)ds, \quad t\in\mathbb R, f\in AP(\mathbb R, \mathbb R^n).
$$
According to Corollary \ref{CO3.20}, the next generation operator of the non-densely defined Cauchy problem \eqref{4.22} is given by 
$$
\mathscr L_*(f)(t)=\mathcal F(t)\lim_{\lambda\to+\infty}\int_{-\infty}^t \overline\Gamma(t,s) e^{\lambda \theta} f(s)ds,\quad \forall t\in \mathbb R,f\in AP(\mathbb R, \mathbb R^n).
$$
By \eqref{4.27}, we have
$$
\begin{array}{ll}
\int_{-\infty}^t \overline \Gamma(t,s) e^{\lambda\theta}f(s)ds&=\int_{-\infty}^{t+\theta} \overline \Gamma(t,s) e^{\lambda\theta}f(s)ds+\int_{t+\theta}^t  \overline \Gamma(t,s) e^{\lambda\theta}f(s)ds\\
&=\int_{-\infty}^{t+\theta}\Pi(t+\theta,s)f(s)ds +\int_{t+\theta}^t e^{\lambda(t-s+\theta)}f(s)ds.
\end{array}
$$
In addition, we have
$$
\lim_{\lambda\to +\infty}\int_{t+\theta}^t e^{\lambda(t-s+\theta)}f(s)ds=0,\quad \forall\theta\in [-\tau,0].
$$
Thus, we can deduce 
$$
\lim_{\lambda\to +\infty}\int_{-\infty}^t \overline \Gamma(t,s) e^{\lambda\cdot}f(s)ds=\int_{-\infty}^{t+\cdot} \Pi(t+\cdot,s)f(s)ds=\int_0^{+\infty} \Pi(t+\cdot,t-s+\cdot)f(t-s+\cdot)ds
$$
Therefore, the next generation operator of the epidemic model \eqref{4.18} can be defined by
$$
\mathscr L_*(f)(t)=\mathcal F(t)\int_0^{+\infty}\Pi(t+\cdot,t-s+\cdot)f(t-s+\cdot)ds, \quad \forall f\in AP(\mathbb R,\mathbb R^n)
$$
and the basic reproduction number $R_0$ is given by 
$$
R_0=r(\mathscr L_*).
$$

\begin{remk}\label{RE4.9}
In \cite{Qiang20}, Qiang, Wang and Zhao studied the almost periodic epidemic models with time delay and gave the definition of the basic reproduction ratio $R_0$. In subsection 4.4, we reconsider this problem by using non-densely defined operators. It is easy to see that our results are the same as \cite{Qiang20}. 
\end{remk}

\section{Ethics declarations}
\subsection{Data availability}
  No data was used in this manuscript.
\subsection{Conflict of interests}
  The authors declare no potential conflict of interests.
\subsection{Ethics approval and consent to participate}
  Not applicable.

\section{Author Contributions statement}
  Jiawei Huo and Rong Yuan wrote the main manuscript text.

\renewcommand{\baselinestretch}{1.2}

\end{document}